\documentclass{amsart}
\usepackage{amsthm}
\usepackage{amsfonts}
\usepackage{graphicx}
\usepackage{pst-3d}

\newtheorem{proposition}{Proposition}
\newtheorem{lemma}[proposition]{Lemma}
\newtheorem{corollary}[proposition]{Corollary}

\makeatletter
\def\xyPlain#1{%
\ThreeDput[normal=0 0 1](0,0,0){\psgrid[subgriddiv=0,gridcolor=lightgray](0,0)(#1,#1)\psline{->}(0,0)(0,#1) \psline{->}(0,0)(#1,0)
\ifdim\psk@gridlabels pt>\z@
\uput[180]{0.2}(0,#1){$y$}\uput[-90]{0.2}(#1,0){$x$}\fi }}
\def\xzPlain#1{\ThreeDput[normal=0 -1 0](0,0,0){\psgrid[subgriddiv=0,gridcolor=lightgray](0,0)(#1,#1) \psline{->}(0,0)(0,#1) \psline{->}(0,0)(#1,0)
\ifdim\psk@gridlabels pt>\z@
\uput[180]{0.2}(0,#1){$z$}\uput[-90]{0.2}(#1,0){$x$}%
\fi }}
\def\yzPlain#1{\ThreeDput[normal=1 0 0](0,0,0){\psgrid[subgriddiv=0,gridcolor=lightgray](0,0)(#1,#1) \psline{->}(0,0)(0,#1) \psline{->}(0,0)(#1,0)
\ifdim\psk@gridlabels pt>\z@
\uput[180]{0.2}(0,#1){$z$}\uput[-90]{0.2}(#1,0){$y$}%
\fi }}
\def\IIIDKOSystem{\@ifnextchar[{\IIIDKOSystem@i}{\IIIDKOSystem@i[]}}
\def\IIIDKOSystem@i[#1]#2{\psset{#1}\xyPlain{#2}\xzPlain{#2}\yzPlain{#2}}
\makeatother

\begin{document}


\title[Numerical method with energy consistency for sine-Gordon]{A numerical method with properties of consistency in the energy domain for a class of dissipative nonlinear wave equations with applications to a Dirichlet boundary-value problem}

\author[J.~E.~Mac\'{\i}as-D\'{\i}az]{J.~E.~Mac\'{\i}as-D\'{\i}az}
\address{Departamento de Matem\'{a}ticas y F\'{\i}sica, Universidad Aut\'{o}noma de Aguascalientes, Avenida Universidad 940, Ciudad Universitaria, Aguascalientes, Ags. 20100, M\'{e}xico}
\email{jemacias@correo.uaa.mx}

\author[A.~Puri]{A.~Puri}
\address{Department of Physics, University of New Orleans, 2000 Lakeshore Drive, New Orleans, LA 70148}
\email{apuri@uno.edu}

\dedicatory{The first author dedicates this work to Prof. L\'{a}szl\'{o} Fuchs}

\begin{abstract}
In this work, we present a conditionally stable finite-difference scheme that consistently approximates the solution of a general class of $(3+1)$-dimensional nonlinear equations that generalizes in various ways the quantitative model governing discrete arrays consisting of coupled harmonic oscillators. Associated with this method, there exists a discrete scheme of energy that consistently approximates its continuous counterpart. The method has the properties that the associated rate of change of the discrete energy consistently approximates its continuous counterpart, and it approximates both a fully continuous medium and a spatially discretized system. Conditional stability of the numerical technique is established, and applications are provided to the existence of the process of nonlinear supratransmission in generalized Klein-Gordon systems and the propagation of binary signals in semi-unbounded, three-dimensional arrays of harmonic oscillators coupled through springs and perturbed harmonically at the boundaries, where the basic model is a modified sine-Gordon equation; our results show that a perfect transmission is achieved via the modulation of the driving amplitude at the boundary. Additionally, we present an example of a nonlinear system with a forbidden band-gap which does not present supratransmission, thus establishing that the existence of a forbidden band-gap in the linear dispersion relation of a nonlinear system is not a sufficient condition for the system to present supratransmission.
\end{abstract}

\maketitle

\section{Introduction}

Almost five years after the appearance of the pioneering letter by Geniet and Leon \cite {Geniet-Leon}, the phenomenon of nonlinear supratransmission has been studied widely in many one-dimensional, physical systems. The phenomenon consists in a sudden increase in the energy injected into a nonlinear system by a harmonic perturbation irradiating at a frequency in the forbidden band-gap, and the research in the field has concentrated mainly on discrete media such as mechanical chains of oscillators described by coupled sine-Gordon and Klein-Gordon equations \cite{Geniet-Leon}, coupled double sine-Gordon equations \cite{Geniet-Leon2}, Fermi-Pasta-Ulam nonlinear chains \cite{Khomeriki}, and Bragg media in the nonlinear Kerr regime \cite{Leon-Spire}. Nonetheless, some research has been done in the continuous case scenario, where the sine-Gordon equation has been a common denominator \cite{Khomeriki-Leon,Chevrieux2}. Meanwhile, from a pragmatic perspective, the importance of the process of nonlinear supratransmision has been evidenced through the many applications proposed to the design of digital amplifiers of ultra weak signals \cite{Khomeriki-Leon2}, light detectors sensitive to very weak excitations \cite{Chevriaux}, optical waveguide arrays \cite{Khomeriki2}, and light filters \cite{Khomeriki-Ruffo}. 

Of course, the problem in the numerical study of the process of nonlinear supratransmission lies in the development of a reliable computational technique to approximate consistently the solutions to the mixed-value problem, the local energy density of the system, and its total energy, in view of the fact that supratransmission is better characterized in the energy domain. Moreover, from a historically point of view the use of symplectic methods for Hamiltonian systems has proved to yield more than satisfactory results \cite{Sanz}; unfortunately, the medium we analyze in the present work contemplates the inclusion of internal and external damping terms which make it nonconservative in general. Nonetheless, the main part of our study will be devoted to develop a finite-difference scheme for the problem under analysis, together with a discrete scheme for the local energy density and the total energy of the system with consistency properties not only in the energy domain, but also in the domain of the rate of change of energy of the medium.

In general, the study of $(3 + 1)$-dimensional systems governed by sine-Gordon equations is an important problem in the physical sciences. For instance, ring-shaped solitary wave solutions of these type of systems have been numerically investigated to show ultimately that such solutions have quasi-soliton properties \cite {Ring-shaped}. The existence of multi-soliton and vortex-soliton solutions has been established for this model, too \cite {Multi-soliton}. $N$-layer sine-Gordon-type models have been studied in order to generalize the results obtained for the two-layer sine-Gordon model \cite {N-layer}, a model that has been used to describe the dynamics of high transition temperature superconductors \cite {Plasma}. Finally, the $(3 + 1)$-dimensional sine-Gordon equation has been used to explore the possibility of stable superluminal propagation of short electromagnetic excitations \cite {Superluminal}.

In Section \ref {Sec2} of this work, we present the $(3 + 1)$-dimensional problem under study in its most general form. The model includes the presence of internal and external damping, relativistic mass, and generalized Josephson currents. Here, we present the Lagrangian of the undamped case as well as an energy analysis of the system under study, and a statement of a similar problem in spherically symmetric media. Section \ref {Sec3} introduces the finite-difference schemes employed to approximate solutions of the mixed-value problem of interest, and the schemes used to approximate the local energy density and the total energy of the system. We establish that the discrete rate of change of energy is a consistent estimate of its continuous counterpart, and a stability condition is proved. In Section \ref {Sec4}, we show numerically that the process of nonlinear supratransmission is present in the semi-discrete system under scrutiny by means of an application of the method presented in this work. The relevance of our results will be shown when we demonstrate next that the system under study does not support supratransmission when the medium is radially perturbed at the origin, whence it will follow the existence of a forbidden band-gap for the frequency in the linear dispersion relation of a nonlinear system does not necessarily guarantee the presence of supratransmission in the medium. A second application to the generation and propagation of localized nonlinear modes in the system of interest is presented next, and we close our work with a section of concluding remarks.

\section{Mathematical models\label{Sec2}}

In the present section we introduce the two mathematical models under study in this work. Throughout, the nonnegative constants $\beta$ and $\gamma$ represent, respectively, the coefficients of internal and external damping of the medium, and the pure-real or pure-imaginary constant $\mathfrak {m}$ denotes a relativistic mass; this last parameter has been included to suggest further applications of our results to the field of particle physics \cite{Kudriatsev}. Moreover, the nonnegative value $J$ will be called \emph {generalized Josephson current}, and its inception has been realized with applications to superconductivity in mind \cite{Solitons}.

\subsection{Cartesian problem}

Let us represent the closure of the first octant of the Euclidean space $\mathbb {R} ^3$ by $D$, let $V$ be any continuously differentiable real function defined in all of $\mathbb {R}$, and assume that $u$ is a function of $( \mathbf {x} , t )$, where $\mathbf {x} \in D$ and $t \in \mathbb {R} ^+$. Under these circumstances, the medium studied in the present paper is described by the generalized partial differential equation
\begin{equation}
\frac {\partial ^2 u} {\partial t ^2} - \nabla ^2 u + \mathfrak {m} ^2 u + V ^\prime (u) - J = \beta \nabla ^2 \left( \frac {\partial u} {\partial t} \right) - \gamma \frac {\partial u} {\partial t}, \label{Eqn:Main}
\end{equation}
in which $\nabla ^2$ represents the Laplacian operator.

It is important to point out that \eqref {Eqn:Main} generalizes nonlinear partial differential equations such as the Klein-Gordon equation, the sine-Gordon equation, and the Landau-Ginzburg equation, amongst others. If $\mathfrak {m}$ is a pure-real number then a modified sine-Gordon model is obtained, for instance, when a potential of the form $V (u) = 1 - \cos u$ is considered, and a modified nonlinear Klein-Gordon equation results when $V (u) = \frac {1} {2 !} u ^2 - \frac {1} {4 !} u ^4$. Meanwhile, for every positive real number $\lambda$, a modified Landau-Ginzburg equation is obtained if $V (u) = \lambda u ^4$ when $\mathfrak {m}$ is a pure-imaginary number. 

This investigation considers particularly the study of the sine-Gordon and Klein-Gordon equations, two models that have been thoroughly studied in the literature \cite {Jorgens,Segal,Morawetz,Glassey,Barone,Scott}. The inclusion of the parameters $\beta$ and $\gamma$ in our model correspond to the need of considering generalizations of physically realistic models in which internal and external damping are present, such as problems arising in the study of long Josephson junctions between superconductors when dissipative effects are taken into account \cite {Solitons} or in the investigation of fluxons in Josephson transmission lines \cite {Lomdahl}. Mathematically, the study of sine-Gordon and Klein-Gordon systems where linear damping is present has lead to the discovery of weak solutions of these equations \cite {Rubino}, the proof of the existence of the maximal attractor in dissipative systems of Klein-Gordon-Schr\"{o}dinger equations \cite {Boling}, the discovery of the mechanism of the ratchet-like dynamics of solitons in dissipative Klein-Gordon media driven by a bi-harmonic force \cite {Morales}, the proof of the existence of multistabilities and soliton trapping in the damped Klein-Gordon equation with external periodic excitation via the asymptotic perturbation method \cite {Maccari}, amongst many other analytical results \cite {Ha,Biler,Pucci,Park,Bahuguna}.

In the case of conservative sine-Gordon and Klein-Gordon media described by \eqref {Eqn:Main} with a generalized Josephson current equal to zero, the linear dispersion relation is obtained when considering solutions in the linearized systems in the form of linear modes \cite {Zauderer}. In such cases, the dispersion relation adopts the form 
\begin{equation}
\omega ^2 (\xi , \zeta , \eta)= \mathfrak {m} ^2 + 1 + 4 \left(\sin ^2 \frac {\xi} {2} + \sin ^2 \frac {\zeta} {2} + \sin ^2 \frac {\eta} {2}\right),
\end{equation}
which possesses a forbidden band-gap given by $\Omega < \sqrt {\mathfrak {m} ^2 + 1}$. From a practical point of view, the parameter $\Omega$ will represent the frequency of the driving boundary in a harmonically perturbed system described by \eqref {Eqn:Main}.

Once the pragmatic importance of our model has been understood, we proceed to simplify it by letting 
\begin{equation}
G (u) = \frac {1} {2} \mathfrak {m} ^2 u ^2 + V (u) - J u.
\end{equation}
In these terms, the Lagrangian associated with the conservative portion of \eqref {Eqn:Main} and the corresponding Hamiltonian are
\begin{equation}
\mathcal {L} = \frac {1} {2} \left\{ \left( \frac {\partial u} {\partial t} \right) ^2 - \Vert \nabla u \Vert ^2 \right\} - G (u) \qquad \text {and} \qquad \mathcal {H} = \frac {1} {2} \left( \frac {\partial u} {\partial t} \right) ^2 + \frac {1} {2} \Vert \nabla u \Vert ^2 + G (u),
\end{equation}
respectively, where $\Vert \cdot \Vert$ represents the Euclidean norm in $\mathbb {R} ^3$. Moreover, it is easy to derive the following expression for the total energy of the system at any time $t$, in which the integrand is the local energy density:
\begin{equation}
E (t) = \iiint _D \mathcal {H} d \mathbf {x}. \label{EnergyEq1}
\end{equation}
Here, it is important to observe that the total energy of our problem is positive whenever $G$ is a nonnegative real function, for instance, when the system has no generalized Josephson current and $V (u) = u ^p$ with $p$ an even positive integer.

For computational reasons, we will restrict our study to bounded domains $D$ of the form $[0 , L] \times [0 , L] \times [0 , L]$, where $L$ is a positive constant. Moreover, we will assume that Neumann boundary data will be imposed on the sides of $D$ opposite to the origin. More precisely, we will assume that $\nabla u \cdot \hat {\mathrm {n}} = 0$ on the sides of $D$ opposite to the origin. Furthermore, it will be important to consider Dirichlet data of the form $u (\mathbf {x} , t) = \phi (t)$, where $\mathbf {x}$ are on the sides of $D$ adjacent to the origin.

\begin{proposition}
The instantaneous rate of change with respect to time of the total energy associated with the partial differential equation \eqref {Eqn:Main} in the region $D = [0 , L] ^3$ of $\mathbb {R} ^3$ with boundary data $\nabla u \cdot \hat {\mathrm {n}} = 0$ on the sides of $D$ opposite to the origin and Dirichlet condition on the sides $\partial D ^+$ adjacent to the origin, is given by
\begin{equation}
\begin{array}{rcl}
E ^\prime (t) & = &  \displaystyle {\iint _{\partial D ^+} \frac {\partial u} {\partial t} \nabla u \cdot \hat {\mathrm {n}} \, d \sigma - \gamma \iiint _D \left( \frac {\partial u} {\partial t} \right) ^2 d \mathbf {x}} \\
 & & \displaystyle {\quad + \beta \left\{ \iint _{\partial D ^+} \frac {\partial u} {\partial t} \frac {\partial} {\partial t} \left( \nabla u \cdot \hat {\mathrm {n}} \right) \, d \sigma - \iiint _D \left\Vert \nabla \left( \frac {\partial u} {\partial t} \right) \right\Vert ^2 d \mathbf {x} \right\}} .
 \end{array} \label {Eqn:DRateE}
\end{equation}
\label{Prop1}
\end{proposition}

\begin{proof} Taking derivative with respect to time on both sides of \eqref{EnergyEq1}, using Green's first identity, and substituting equation \eqref {Eqn:Main} next, we obtain that 
\begin{eqnarray}
E ^\prime (t) & = & \iiint _D \frac {\partial u} {\partial t} \left\{ \frac {\partial ^2 u} {\partial t^2} + G ^\prime (u) \right\} d \mathbf {x} + \frac {1} {2} \iiint _D \frac {\partial} {\partial t} \Vert \nabla u \Vert ^2 d \mathbf {x} \nonumber \\
 & = & \iiint _D \frac {\partial u} {\partial t} \left\{ \frac {\partial ^2 u} {\partial t^2} - \nabla ^2 u + G ^\prime (u) \right\} d \mathbf {x} + \iint _{\partial D} \frac {\partial u} {\partial t} \nabla u \cdot \hat {\mathrm {n}} \, d \sigma \nonumber \\
 & = & \beta \iiint _D \frac {\partial u} {\partial t} \nabla ^2 \left( \frac {\partial u} {\partial t} \right) \, d \mathbf {x} - \gamma \iiint _D \left( \frac {\partial u} {\partial t} \right) ^2 d \mathbf {x} + \iint _{\partial D} \frac {\partial u} {\partial t} \nabla u \cdot \hat {\mathrm {n}} \, d \sigma. \nonumber
\end{eqnarray}
On the other hand, from Green's first identity we see that
\begin{equation}
\iiint _D \frac {\partial u} {\partial t} \nabla ^2 \left( \frac {\partial u} {\partial t} \right) \, d \mathbf {x} = \iint _{\partial D} \frac {\partial u} {\partial t} \frac {\partial} {\partial t} \left( \nabla u \cdot \hat {\mathrm {n}} \right) \, d \sigma - \iiint _D \left\Vert \nabla \left( \frac {\partial u} {\partial t} \right) \right\Vert ^2 d \mathbf {x}. \nonumber
\end{equation}
The surface integrals in these last two equations are equal to zero on the three sides of $D$ opposite to the origin, whence the result follows.
\end{proof}

It is worth noticing that, in view of the hypotheses of Proposition \ref {Prop1}, more concrete expressions for some terms in \eqref {Eqn:DRateE} are readily at hand. Particularly, it is convenient to observe that  
\begin{equation}
\begin{array}{rcl}
\displaystyle {\iint _{\partial D ^+} \frac {\partial u} {\partial t} \nabla u \cdot \hat {\mathrm {n}} \, d \sigma} & = & - \displaystyle {\int _0 ^L \int _0 ^L \frac {\partial u} {\partial t} (0 , y , z) \frac {\partial u} {\partial x} (0 , y , z) \, d y \, d z} \\
 & & \displaystyle {\quad - \int _0 ^L \int _0 ^L \frac {\partial u} {\partial t} (x , 0 , z) \frac {\partial u} {\partial y} (x , 0 , z) \, d x \, d z} \\
 & & \displaystyle {\qquad - \int _0 ^L \int _0 ^L \frac {\partial u} {\partial t} (x , y , 0) \frac {\partial u} {\partial z} (x , y , 0) \, d x \, d y}.
\end{array}
\end{equation}
Similarly, 
\begin{equation}
\begin{array}{rcl}
\displaystyle {\iint _{\partial D ^+} \frac {\partial u} {\partial t} \frac {\partial} {\partial t} (\nabla u \cdot \hat {\mathrm {n}}) \, d \sigma} & = & - \displaystyle {\int _0 ^L \int _0 ^L \frac {\partial u} {\partial t} (0 , y , z) \frac {\partial ^2 u} {\partial t \, \partial x} (0 , y , z) \, d y \, d z} \\
 & & \displaystyle {\quad - \int _0 ^L \int _0 ^L \frac {\partial u} {\partial t} (x , 0 , z) \frac {\partial ^2 u} {\partial t \, \partial y} (x , 0 , z) \, d x \, d z} \\
 & & \displaystyle {\qquad - \int _0 ^L \int _0 ^L \frac {\partial u} {\partial t} (x , y , 0) \frac {\partial ^2 u} {\partial t \, \partial z} (x , y , 0) \, d x \, d y}.
\end{array}
\end{equation}
It is also important to notice that if $\beta = \gamma = 0$ and if either $\frac {\partial u} {\partial t} = 0$ or $\nabla u \cdot \hat {\mathrm {n}} = 0$ on $\partial D ^+$, then the energy of the system is conserved throughout time.

\subsection{Spherical problem}

As usual, let $D$ be the closure of the first octant in $\mathbb {R} ^3$, and assume that $u$ is a radially symmetric solution of \eqref {Eqn:Main}. Let $r = \Vert \mathbf {x} \Vert$ represent the Euclidean norm of the vector $\mathbf {x} \in \mathbb {R} ^3$, let $v (r , t) = r u (r , t)$, and assume that $u$ is a solution of problem \eqref {Eqn:Main} for Dirichlet boundary data in the origin given by $u (\mathbf {0} , t) = \phi (t)$. Then $v$ satisfies the relation $v (0 , t) = 0$, together with the partial differential equation 
\begin{equation}
\frac {\partial ^2 v} {\partial t ^2} - \frac {\partial ^2 v} {\partial r ^2} + \mathfrak {m} ^2 v + r V ^\prime \left( \frac {v} {r} \right) - J r = \beta \frac {\partial ^3 v} {\partial t \, \partial r ^2} - \gamma \frac {\partial v} {\partial t}. \label{Eqn:MainRad}
\end{equation}

Computationally and for the remainder of the present section, the region $D$ will represent the closure of the portion of the solid sphere with center in the origin and radius equal to $L$ that lies in the first octant, and Neumann boundary data will be imposed on the boundary of the region. Moreover, since the Dirichlet boundary condition of $u$ at the origin translates into a void condition for $v$ (which in turn translates into a trivial solution for problem \eqref {Eqn:MainRad} when vanishing initial conditions are chosen and the Josephson current is equal to zero), we set $v (\epsilon , t ) = \epsilon \phi (t)$ for some $\epsilon > 0$ sufficiently close to zero.

Under the presence of spherical symmetry and assuming that $u (L , t) = 0$ for every $t \in \mathbb {R} ^+$, the energy expression \eqref {EnergyEq1} of the undamped system may be computed in terms of $v$ via the following expression:
\begin{equation}
E (t) = \frac {\pi} {2} \int _0 ^L \left\{ \frac {1} {2} \left( \frac {\partial v} {\partial t} \right) ^2 + \frac {1} {2} \left( \frac {\partial v} {\partial r} \right) ^2 + r ^2 G \left( \frac {v} {r} \right) \right\} \, d r. \label {Eqn:RadEnergy}
\end{equation}
This follows immediately after noticing that $\left( \frac {\partial v} {\partial r} \right) ^2 = \left( r \frac {\partial u} {\partial r} \right) ^2 + \frac {\partial} {\partial r} (r u ^2)$. Here, $G$ is the simplified form of the potential function provided in the previous subsection. Moreover, in view of Proposition \ref {Prop1}, a rate of change of the total energy of system \eqref {Eqn:MainRad} is readily established as our next analytical result.

\begin{proposition}
The instantaneous rate of change of the total energy of a system satisfying \eqref {Eqn:MainRad} on the set $D$, with Dirichlet boundary condition at the origin, and Neumann data $\nabla u \cdot \hat {\mathrm {n}} = 0$ and Dirichlet data $u (L , t) = 0$ in the intersection of $D$ with the set of $\mathbf {x} \in \mathbb {R} ^3$ such that $\Vert \mathbf {x} \Vert = L$, is provided by the formula
\begin{equation}
E ^\prime (t) = - \frac {\pi} {2} \int _0 ^L \left\{ \beta \left( \frac {\partial ^2 v} {\partial t \, \partial r} - \frac {1} {r} \frac {\partial v} {\partial t} \right) ^2 + \gamma \left( \frac {\partial v} {\partial t} \right) ^2 \right\} \, d r.
\end{equation}
\end{proposition}

\begin{proof}
It follows directly from Proposition \ref {Prop1} and the substitution $v (r , t) = r u (r , t)$.
\end{proof}

In this case we must observe that the Neumann boundary data will take the form $\frac {\partial u} {\partial r} = 0$ on the curved side of the wedge $D$. In terms of the variable $v$, this condition translates into the equation
\begin{equation}
\frac {\partial v} {\partial r} + \frac {v} {r} = 0. \label{Eqn:BoundRad}
\end{equation}

\subsection{Discrete problem}

In this section, we introduce a model that describes the dynamics of a discrete system of pendula attached springs. Let $u _{m , n , p}$ be a real function on the real variable $t$, for every $m , n , p \in \mathbb {Z} ^+ \cup \{ 0 \}$ and every $t \geq 0$. We will consider now the infinite system of coupled ordinary differential equations with constant coupling coefficient $c > 0$, in which $m , n , p \in \mathbb {Z} ^+$:
\begin{equation}
\ddot {u} _{m , n , p} - c ^2 \nabla ^2 u _{m , n , p} + \mathfrak {m} ^2 u _{m , n , p} + V ^\prime ( u _{m , n , p} ) - J = \beta \nabla ^2 \dot {u} _{m , n , p} - \gamma \dot {u} _{m , n , p} \label{Eqn:DiscreteMain}
\end{equation}
Here, the discrete Laplacian operator $\nabla ^2$, the Hamiltonian $H$ of the lattice site at position $(m , n , p)$ for the conservative theory, and the total energy $E$ of the system are, respectively,
\begin{eqnarray}
\nabla ^2 u _{m , n , p} & = & \displaystyle {u _{m + 1 , n , p} + u _{m - 1 , n , p} + u _{m , n + 1 , p} + u _{m , n - 1 , p} + u _{m , n , p + 1} + u _{m , n , p - 1} - 6 u _{m , n , p},} \nonumber \\
H _{m , n , p} & = & \displaystyle {\frac {1} {2} \left\{ \dot {u} _{m , n , p} ^2 + c ^2 (u _{m + 1 , n , p} - u _{m , n , p}) ^2 + c ^2 (u _{m , n + 1 , p} - u _{m , n , p}) ^2 \right.} \nonumber \\
 & & \displaystyle {\quad \left. + c ^2 (u _{m , n , p + 1} - u _{m , n , p}) ^2 + \mathfrak {m} ^2 u _{m , n , p} ^2 \right\} + V (u _{m , n , p}) - J u _{m , n , p},}  \\
E & = & \displaystyle {\sum _{m , n , p = 1} ^N H _{m , n , p} + \frac {c ^2} {2} \left[\sum _{n , p = 1} ^N (u _{1 , n , p} - u _{0 , n , p}) ^2 + \sum _{m , p = 1} ^N (u _{m , 1 , p} - u _{m , 0 , p}) ^2 \right.} \nonumber \\
 & & \displaystyle {\quad \left. + \sum _{m , n = 1} ^N (u _{m , n , 1} - u _{m , n , 0}) ^2 \right].} \nonumber
\end{eqnarray}
It is important to remark here that the inclusion of the terms multiplied by $\frac {c ^2} {2}$ in the discrete energy corresponds to the need to include the potential from the coupling of the nodes adjacent to the boundary.

For the sake of convenience, we introduce the notation
\begin{equation}
\begin{array}{rcl}
\delta _x u _{m , n , p} & = & u _{m + 1 , n , p} - u _{m , n , p},\\ 
\delta _y u _{m , n , p} & = & u _{m , n + 1 , p} - u _{m , n , p}, \\
\delta _z u _{m , n , p} & = & u _{m , n , p + 1} - u _{m , n , p}. 
\end{array}
\end{equation}
Moreover, for computational reasons we will assume that $m$, $n$ and $p$ take on values in the set $\{ 0 , 1 , \dots , N + 1 \}$ for a relatively large positive integer $N$, and assume that discrete Neumann boundary data are imposed on the boundaries $n = N + 1$, $m = N + 1$ and $p = N + 1$, that is, we assume that 
\begin{equation}
\delta _x u _{N , m , p} = \delta _y u _{m , N , p} = \delta _z u _{m , n , N} = 0,
\end{equation}
for every $m , n , p \in \{ 1 , \dots , N \}$. Meanwhile, Dirichlet data will be required on the remaining boundaries. 

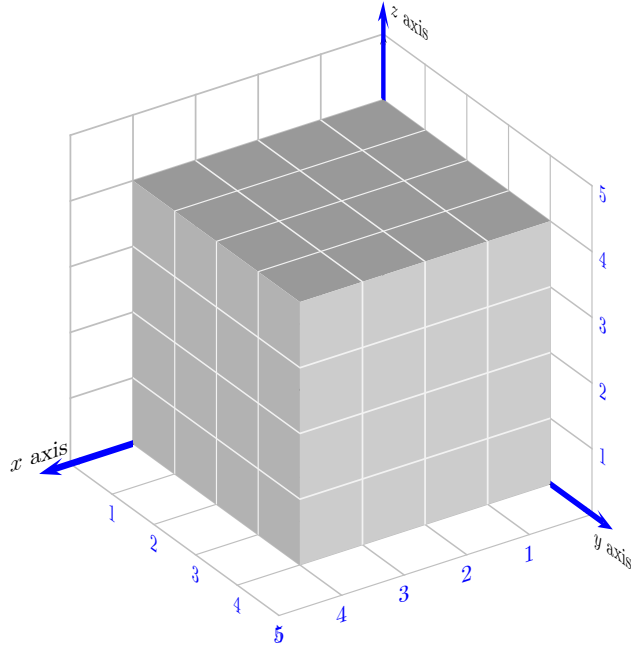
\begin{figure}
\centerline{%
\begin{pspicture}(-4.5,-3.5)(3,4.75)
\newgray{gray75}{0.6}\newgray{gray85}{0.7}\newgray{gray95}{0.8}\newgray{gray05}{0.95}
\psset{viewpoint=1 1.5 1}
\IIIDKOSystem[gridlabels=0pt,gridcolor=lightgray,subgriddiv=0]{5}%
\ThreeDput[normal=1 0 0]{\psline[linewidth=3pt,linecolor=blue]{->}(4,0)(5.5,0)\uput[270](5.5,0){\psscalebox{1 1}{\color{black}$y$ axis}}}%
\ThreeDput[normal=0 -1 0]{\psline[linewidth=3pt,linecolor=blue]{->}(4,0)(5.5,0)\uput[90](5.5,0){\psscalebox{-1 1}{\textcolor{black}{$x$ axis}}}}%
\ThreeDput[normal=1 0 0]{\psline[linewidth=3pt,linecolor=blue]{->}(0,4)(0,5.5)\uput[0](0,5.5){$z$ axis}}%
\ThreeDput[normal=0 0 1](0,0,4){\psframe*[linecolor=gray75](4,4)\rput(2,2){}}%
\ThreeDput[normal=0 1 0](4,4,0){\psframe*[linecolor=gray95](4,4)\rput(2,2){}}%
\ThreeDput[normal=1 0 0](4,0,0){\psframe*[linecolor=gray85](4,4)\rput(2,2){}}%
\ThreeDput[normal=0 0 1](0,0,4){\psline[linecolor=gray05](4,1)(0,1)\psline[linecolor=gray05](1,4)(1,0)%
	\psline[linecolor=gray05](4,2)(0,2)\psline[linecolor=gray05](2,4)(2,0)%
	\psline[linecolor=gray05](4,3)(0,3)\psline[linecolor=gray05](3,4)(3,0)}%
\ThreeDput[normal=0 1 0](4,4,0){\psline[linecolor=gray05](4,1)(0,1)\psline[linecolor=gray05](1,4)(1,0)%
	\psline[linecolor=gray05](4,2)(0,2)\psline[linecolor=gray05](2,4)(2,0)%
	\psline[linecolor=gray05](4,3)(0,3)\psline[linecolor=gray05](3,4)(3,0)}%
\ThreeDput[normal=1 0 0](4,0,0){\psline[linecolor=gray05](4,1)(0,1)\psline[linecolor=gray05](1,4)(1,0)%
	\psline[linecolor=gray05](4,2)(0,2)\psline[linecolor=gray05](2,4)(2,0)%
	\psline[linecolor=gray05](4,3)(0,3)\psline[linecolor=gray05](3,4)(3,0)}%
\ThreeDput[normal=0 1 0](1,5,0){\uput[270](0,0){\color{blue}$1$}}%
\ThreeDput[normal=0 1 0](2,5,0){\uput[270](0,0){\color{blue}$2$}}%
\ThreeDput[normal=0 1 0](3,5,0){\uput[270](0,0){\color{blue}$3$}}%
\ThreeDput[normal=0 1 0](4,5,0){\uput[270](0,0){\color{blue}$4$}}%
\ThreeDput[normal=0 1 0](5,5,0){\uput[270](0,0){\color{blue}$5$}}%
\ThreeDput[normal=1 0 0](5,1,0){\uput[270](0,0){\color{blue}$1$}}%
\ThreeDput[normal=1 0 0](5,2,0){\uput[270](0,0){\color{blue}$2$}}%
\ThreeDput[normal=1 0 0](5,3,0){\uput[270](0,0){\color{blue}$3$}}%
\ThreeDput[normal=1 0 0](5,4,0){\uput[270](0,0){\color{blue}$4$}}%
\ThreeDput[normal=1 0 0](5,5,0){\uput[270](0,0){\color{blue}$5$}}%
\ThreeDput[normal=1 0 0](0,5,1){\uput[0](0,0){\color{blue}$1$}}%
\ThreeDput[normal=1 0 0](0,5,2){\uput[0](0,0){\color{blue}$2$}}%
\ThreeDput[normal=1 0 0](0,5,3){\uput[0](0,0){\color{blue}$3$}}%
\ThreeDput[normal=1 0 0](0,5,4){\uput[0](0,0){\color{blue}$4$}}%
\ThreeDput[normal=1 0 0](0,5,5){\uput[0](0,0){\color{blue}$5$}}%
\end{pspicture}}%
\caption{Schematic representation of a medium governed by system \eqref {Eqn:DiscreteMain}. The nodes located in sites $(m , n , p)$ for $m , n , p \in \mathbb {Z} ^+ \cup \{ 0 \}$ may physically represent harmonic oscillators, and the attaching segments of line play the roles of springs. \label{Fig0}}
\end{figure}

System \eqref {Eqn:DiscreteMain} describes the evolution of a semi-unbounded, three-dimensional array of harmonic oscillators coupled through identical springs with a coupling coefficient equal to $c$. The pendula are located at the discrete sites $(m , n , p)$, where $m , n , p \in \mathbb {Z} ^+$, and the springs are parallel to a coordinate axis. Evidently, site $(m , n , p)$ is coupled with the six sites $(m \pm 1 , n , p)$, $(m , n \pm 1, p)$ and $(m , n , p \pm 1)$, and $c$ represents the common coupling coefficient. Moreover, in practice we will subject the oscillators on the boundaries to harmonic driving in the form of the Dirichlet conditions 
\begin{equation}
u _{m , n , 0} = u _{m , 0 , p} = u _{0 , n , p} = A \sin (\Omega t),
\end{equation}
where $\Omega$ is a frequency in the forbidden band-gap of the continuous-limit medium. A schematic representation of such a system is depicted in Fig. \ref {Fig0}. In this context, it is important to notice that the term with coefficient $\frac {c ^2} {2}$ in the expression $E$ for the total energy of the system corresponds to the potential energy due to the coupling to the driving boundary.

Our next result is a one-dimensional version of Green's first identity. For a proof, we refer to \cite {Macias-Supra}.

\begin{lemma} [Green's discrete first identity]
For every sequence $(a _n) _{n = 0} ^{N + 1}$ for which $a _{N + 1} = a _N$,
$$
\sum _{n = 1} ^N (a _{n + 1} - 2 a _n + a _{n - 1}) a _n = a _0 (a _0 - a _1) - \sum _{n = 1} ^N (a _n - a _{n - 1}) ^2. \qedhere\qed
$$ \label{Lemma:Green}
\end{lemma}

\begin{proposition}
Consider a system satisfying \eqref {Eqn:DiscreteMain} for $m , n , p = 1 , 2 , \dots , N + 1$, subject to discrete Neumann conditions of the form $\delta _x u _{N , m , p} = \delta _y u _{m , N , p} = \delta _z u _{m , n , N} = 0$ on the boundaries $n = N + 1$, $m = N + 1$, and $p = N + 1$, and subject to Dirichlet data on the remaining boundaries. Then, the instantaneous rate of change of the energy of the node in site $(m , n , p)$ with respect to time is given by 
\begin{eqnarray}
\frac {d E} {d t} & = & - c ^2 \sum _{i , j = 1} ^N \left[ (\delta _x u _{0 , i , j}) \dot {u} _{0 , i , j} + (\delta _y u _{i , 0 , j}) \dot {u} _{i , 0 , j} + (\delta _z u _{i , j , 0}) \dot {u} _{i , j , 0} \right] \nonumber \\
 & & \quad - \beta \left\{ \sum _{m , n , p = 1} ^N \left[ (\delta _x \dot {u} _{m - 1 , n , p}) ^2 + (\delta _y \dot {u} _{m , n - 1 , p}) ^2 + (\delta _z \dot {u} _{m , n , p - 1}) ^2 \right] \right. \\
 & & \qquad \left. + \sum _{i , j = 1} ^N \left[ (\delta _x \dot {u} _{0 , i , j}) \dot {u} _{0 , i , j} + (\delta _y \dot {u} _{i , 0 , j}) \dot {u} _{i , 0 , j} + (\delta _z \dot {u} _{i , j , 0}) \dot {u} _{i , j , 0} \right] \right\} - \gamma \sum _{m , n , p = 1} ^N (\dot {u} _{m , n , p}) ^2. \nonumber
\end{eqnarray} \label{Prop:2-4}
\end{proposition}

\begin{proof}
Define $K _{m , n , p} = - c ^2 \dot {u} _{m , n , p} (u _{m , n , p} - u _{m - 1 , n , p})$, $K _{m , n , p} ^\prime = - c ^2 \dot {u} _{m , n , p} (u _{m , n , p} - u _{m , n - 1 , p})$ and $K _{m , n , p} ^{\prime \prime} = - c ^2 \dot {u} _{m , n , p} (u _{m , n , p} - u _{m , n , p - 1})$, for every $m , n , p \in \{1 , 2 , \dots , N\}$. It is necessary to observe first of all that 
\begin{equation*}
\frac {c ^2} {2} \frac {d} {d t} (u _{m + 1 , n , p} - u _{m , n , p}) ^2 = \left( K _{m , n , p} - K _{m + 1 , n , p} \right) - c ^2 (u _{m + 1 , n , p} - 2 u _{m , n , p} + u _{m - 1 , n , p}) \dot {u} _{m , n , p}.
\end{equation*}
Similar relations may be obtained for the derivatives of the other terms in the Hamiltonian which are multiplied by $c ^2$. Moreover, taking derivative of the Hamiltonian with respect to $t$, one obtains that 
\begin{eqnarray*}
\frac {d H _{m , n , p}} {d t} & = & (K _{m , n , p} - K _{m + 1 , n , p}) + (K _{m , n , p} ^\prime - K _{m , n + 1 , p} ^\prime) + (K _{m , n , p} ^{\prime \prime} - K _{m , n , p + 1} ^{\prime \prime}) \\
 & & + \left\{ \ddot {u} _{m , n , p} - c ^2 \nabla ^2 u _{m , n , p} + \mathfrak {m} ^2 u _{m , n , p} + V ^\prime (u _{m , n , p}) - J \right\} \dot {u} _{m , n , p} \\
 & = & (K _{m , n , p} - K _{m + 1 , n , p}) + (K _{m , n , p} ^\prime - K _{m , n + 1 , p} ^\prime) + (K _{m , n , p} ^{\prime \prime} - K _{m , n , p + 1} ^{\prime \prime}) \\
 & & + \beta (\dot {u} _{m + 1 , n , p} - 2 \dot {u} _{m , n , p} + \dot {u} _{m - 1 , n , p}) \dot {u} _{m , n , p} + \beta (\dot {u} _{m , n + 1 , p} - 2 \dot {u} _{m , n , p} + \dot {u} _{m , n - 1 , p}) \dot {u} _{m , n , p} \\
 & & + \beta (\dot {u} _{m , n , p + 1} - 2 \dot {u} _{m , n , p} + \dot {u} _{m , n , p - 1}) \dot {u} _{m , n , p} - \gamma (\dot {u} _{m , n , p}) ^2,
\end{eqnarray*} 
and sum over indexes of $m$, $n$ and $p$ in the set $\{ 1 , 2 , \dots , N \}$. We identify the sums of the first three expressions in parenthesis as telescoping series and proceed to simplify; at the same time, three applications of the discrete version of Green's first identity provide alternative expressions for the terms multiplied by $\beta$. On the other hand, by differentiating the energy expression with respect to time, substituting the derivative of the Hamiltonians and simplifying, we reach the desired formula.
\end{proof}

\section{Numerical analysis\label{Sec3}}

\subsection{Cartesian problem \label{Sec:Thm}}

In order to approximate solutions of the partial differential equation \eqref {Eqn:Main} on the cube $[0 , L] \times [0 , L] \times [0 , L]$ over an interval of time of length $T$, we choose a regular partition $0 = t _0 < t _1 < \dots < t _M = T$ of $[0 , T]$ with time step equal to $\Delta t$, as well as three regular partitions of $[0 , L]$ consisting of $N _x + 1$, $N _y + 1$ and $N _z + 1$ subintervals, each with step equal to $\Delta x$, $\Delta y$ and $\Delta z$, respectively. For all permissible indexes $k$, $m$, $n$ and $p$, we represent the approximate solution to our problem at time $k \Delta t$ and at the location $(m \Delta x , n \Delta y , p \Delta z)$ by $u ^k _{m , n , p}$. The discretization of the problem under study is provided by the finite-difference schemes
\begin{equation}
\begin{array}{c}
\begin{array}{rcl}
\displaystyle {\frac {\delta ^2 _t u _{m , n , p} ^k} {(\Delta t) ^2} - \left( 1 + \beta \frac {\delta _t} {2 \Delta t} \right) \left[ \frac {\delta ^2 _x} {(\Delta x) ^2} + \frac {\delta _y ^2} {(\Delta y) ^2} + \frac {\delta _z ^2} {(\Delta z) ^2} \right] u _{m , n , p} ^k + \frac {\gamma \delta _t} {2 \Delta t} u _{m , n , p} ^k + \quad} & & \\
\displaystyle {\frac {\mathfrak {m} ^2} {2} [u _{m , n , p} ^{k + 1} + u _{m , n , p} ^{k - 1}] + \frac {V (u _{m , n , p} ^{k + 1}) - V (u _{m , n , p} ^{k - 1})} {u _{m , n , p} ^{k + 1} - u _{m , n , p} ^{k - 1}}} - J & = & 0,
\end{array} \label{Eqn:DiffEq2}
\end{array}
\end{equation}
for every $k = 1 , \dots , M - 1$, $m = 1 , \dots , N _x$, $n = 1 , \dots , N _y$ and $p = 1 , \dots N _z$, subject to the conditions 
\begin{equation}
\delta _x u _{N _x , n , p} = \delta _y u _{m , N _y , p} = \delta _z u _{m , n , N _z} = 0. \label{Eqn:BoundCond}
\end{equation}
Here, part of the following notation has been employed for the sake of simplicity:
\begin{equation}
\begin{array}{rclrcl}
\delta _t u _{m , n , p} ^k & = & u _{m , n , p} ^{k + 1} - u _{m , n , p} ^{k - 1}, & \delta ^2 _t u _{m ,n , p} ^k & = & u _{m , n , p} ^{k + 1} - 2 u _{m , n , p} ^k + u _{m , n , p} ^{k - 1}, \\
\delta ^2 _x u _{m , n , p} ^k & = & u _{m + 1 , n , p} ^k - 2 u _{m ,n , p} ^k + u _{m - 1 , n , p} ^k, & \delta ^2 _y u _{m , n , p} ^k & = & u _{m , n + 1 , p} ^k - 2 u _{m ,n , p} ^k + u _{m , n - 1 , p} ^k, \\
\delta ^2 _z u _{m , n , p} ^k & = & u _{m , n , p + 1} ^k - 2 u _{m ,n , p} ^k + u _{m , n , p - 1} ^k.
\end{array}
\end{equation}
The forward-difference stencil of the method is presented in Fig. \ref {Fig0-5} for convenience. Moreover, we introduce the composite operators $\delta _{t x} = \delta _x \delta _t$, $\delta _{t y} = \delta _y \delta _t$, and $\delta _{t z} = \delta _z \delta _t$, and the constant $\Delta \upsilon = \Delta x \Delta y \Delta z$. In these terms, the Hamiltonian of the lattice site at position $(m , n , p)$ and the total energy of the system are, respectively,
\begin{equation}
\begin{array}{rcl}
H _{m , n , p} ^k & = & \displaystyle {\frac {1} {2} \left( \frac {u _{m , n , p} ^{k + 1} - u _{m , n , p} ^k} {\Delta t} \right) ^2 + \frac {1} {2} \left[ \frac {(\delta _x u _{m , n , p} ^{k + 1}) (\delta _x u _{m , n , p} ^k)} {(\Delta x) ^2} + \frac {(\delta _y u _{m , n , p} ^{k + 1}) (\delta _y u _{m , n , p} ^k)} {(\Delta y) ^2} \right.} \\
 & & \displaystyle {\quad \left. + \frac {(\delta _z u _{m , n , p} ^{k + 1}) (\delta _z u _{m , n , p} ^k)} {(\Delta z) ^2} \right] + \frac {\mathfrak {m} ^2} {2} \frac {(u _{m , n , p} ^{k + 1}) ^2 + (u _{m , n , p} ^k) ^2} {2}} \\
 & & \displaystyle {\qquad + \frac {V (u _{m , n , p} ^{k + 1}) + V (u _{m , n , p} ^k)} {2} - J \frac {u _{m , n , p} ^{k + 1} + u _{m , n , p} ^k} {2},} \\
\displaystyle {E ^k} & = & \displaystyle {\sum _{m , n , p = 1} ^N H _{m , n , p} ^k \, \Delta \upsilon + \frac {1} {2} \sum _{i , j = 1} ^N \left[ \frac {(\delta _x u _{0 , i , j} ^{k + 1}) (\delta _x u _{0 , i , j} ^k)} {(\Delta x) ^2} + \frac {(\delta _y u _{i , 0 , j} ^{k + 1}) (\delta _y u _{i , 0 , j} ^k)} {(\Delta y) ^2} \right.} \\
 & & \displaystyle {\quad \left.  + \frac {(\delta _z u _{i , j , 0} ^{k + 1}) (\delta _z u _{i , j , 0} ^k)} {(\Delta z) ^2} \right] \Delta \upsilon.}
\end{array}
\end{equation}

\begin{figure}
\centerline{%
\begin{tabular}{cccc}
 & $t _{k - 1}$ & $t _k$ & $t _{k + 1}$\\ &
\begin{pspicture}(-2.5,-1.5)(2.5,3)
\newgray{gray75}{0.6}\newgray{gray85}{0.7}\newgray{gray95}{0.8}\newgray{gray05}{0.95}
\psset{viewpoint=1 1.7 0.73}
\IIIDKOSystem[gridlabels=0pt,gridcolor=lightgray,subgriddiv=0]{3}%
\ThreeDput[normal=1 0 0]{\psline[linewidth=3pt,linecolor=blue]{->}(0,0)(3,0)}%
\ThreeDput[normal=0 -1 0]{\psline[linewidth=3pt,linecolor=blue]{->}(0,0)(3,0)}%
\ThreeDput[normal=1 0 0]{\psline[linewidth=3pt,linecolor=blue]{->}(0,0)(0,3)}%
\ThreeDput[normal=0 0 1](1,1,3){\psframe*[linecolor=gray95](2,2)}%
\ThreeDput[normal=0 0 1](1,1,2){\psframe*[linecolor=gray95](2,2)}%
\ThreeDput[normal=0 0 1](1,1,1){\psframe*[linecolor=gray95](2,2)}%
\ThreeDput[normal=0 0 1](0,0,3){\psline[linecolor=gray05](2,1)(2,3)\psline[linecolor=gray05](1,2)(3,2)%
	\rput(2,2){$\circ$}}
\ThreeDput[normal=0 0 1](0,0,2){\psline[linecolor=gray05](2,1)(2,3)\psline[linecolor=gray05](1,2)(3,2)%
	\rput(2,1){$\circ$}\rput(2,2){$\circ$}\rput(2,3){$\circ$}\rput(3,2){$\circ$}\rput(1,2){$\circ$}}%
\ThreeDput[normal=0 0 1](0,0,1){\psline[linecolor=gray05](2,1)(2,3)\psline[linecolor=gray05](1,2)(3,2)%
	\rput(2,2){$\circ$}}
\ThreeDput[normal=1 0 0](1,3,0){\uput[0](0,0){\color{blue}$n-1$}}%
\ThreeDput[normal=1 0 0](2,3,0){\uput[0](0,0){\color{blue}$n$}}%
\ThreeDput[normal=1 0 0](3,3,0){\uput[0](0,0){\color{blue}$n+1$}}%
\ThreeDput[normal=0 1 0](3,1,0){\uput[180](0,0){\color{blue}$m-1$}}%
\ThreeDput[normal=0 1 0](3,2,0){\uput[180](0,0){\color{blue}$m$}}%
\ThreeDput[normal=0 1 0](3,3,0){\uput[180](0,0){\color{blue}$m+1$}}%
\ThreeDput[normal=1 0 0](0,3,1){\uput[0](0,0){\color{blue}$p-1$}}%
\ThreeDput[normal=1 0 0](0,3,2){\uput[0](0,0){\color{blue}$p$}}%
\ThreeDput[normal=1 0 0](0,3,3){\uput[0](0,0){\color{blue}$p+1$}}%
\end{pspicture} & 
\begin{pspicture}(-2.5,-1.5)(2.5,3)
\newgray{gray75}{0.6}\newgray{gray85}{0.7}\newgray{gray95}{0.8}\newgray{gray05}{0.95}
\psset{viewpoint=1 1.7 0.73}
\IIIDKOSystem[gridlabels=0pt,gridcolor=lightgray,subgriddiv=0]{3}%
\ThreeDput[normal=1 0 0]{\psline[linewidth=3pt,linecolor=blue]{->}(0,0)(3,0)}%
\ThreeDput[normal=0 -1 0]{\psline[linewidth=3pt,linecolor=blue]{->}(0,0)(3,0)}%
\ThreeDput[normal=1 0 0]{\psline[linewidth=3pt,linecolor=blue]{->}(0,0)(0,3)}%
\ThreeDput[normal=0 0 1](1,1,3){\psframe*[linecolor=gray95](2,2)}%
\ThreeDput[normal=0 0 1](1,1,2){\psframe*[linecolor=gray95](2,2)}%
\ThreeDput[normal=0 0 1](1,1,1){\psframe*[linecolor=gray95](2,2)}%
\ThreeDput[normal=0 0 1](0,0,3){\psline[linecolor=gray05](2,1)(2,3)\psline[linecolor=gray05](1,2)(3,2)%
	\rput(2,2){$\circ$}}
\ThreeDput[normal=0 0 1](0,0,2){\psline[linecolor=gray05](2,1)(2,3)\psline[linecolor=gray05](1,2)(3,2)%
	\rput(2,1){$\circ$}\rput(2,2){$\circ$}\rput(2,3){$\circ$}\rput(3,2){$\circ$}\rput(1,2){$\circ$}}%
\ThreeDput[normal=0 0 1](0,0,1){\psline[linecolor=gray05](2,1)(2,3)\psline[linecolor=gray05](1,2)(3,2)%
	\rput(2,2){$\circ$}}
\ThreeDput[normal=1 0 0](1,3,0){\uput[0](0,0){\color{blue}$n-1$}}%
\ThreeDput[normal=1 0 0](2,3,0){\uput[0](0,0){\color{blue}$n$}}%
\ThreeDput[normal=1 0 0](3,3,0){\uput[0](0,0){\color{blue}$n+1$}}%
\ThreeDput[normal=0 1 0](3,1,0){\uput[180](0,0){\color{blue}$m-1$}}%
\ThreeDput[normal=0 1 0](3,2,0){\uput[180](0,0){\color{blue}$m$}}%
\ThreeDput[normal=0 1 0](3,3,0){\uput[180](0,0){\color{blue}$m+1$}}%
\ThreeDput[normal=1 0 0](0,3,1){\uput[0](0,0){\color{blue}$p-1$}}%
\ThreeDput[normal=1 0 0](0,3,2){\uput[0](0,0){\color{blue}$p$}}%
\ThreeDput[normal=1 0 0](0,3,3){\uput[0](0,0){\color{blue}$p+1$}}%
\end{pspicture} & 
\begin{pspicture}(-2.5,-1.5)(2.5,3)
\newgray{gray75}{0.6}\newgray{gray85}{0.7}\newgray{gray95}{0.8}\newgray{gray05}{0.95}
\psset{viewpoint=1 1.7 0.73}
\IIIDKOSystem[gridlabels=0pt,gridcolor=lightgray,subgriddiv=0]{3}%
\ThreeDput[normal=1 0 0]{\psline[linewidth=3pt,linecolor=blue]{->}(0,0)(3,0)}%
\ThreeDput[normal=0 -1 0]{\psline[linewidth=3pt,linecolor=blue]{->}(0,0)(3,0)}%
\ThreeDput[normal=1 0 0]{\psline[linewidth=3pt,linecolor=blue]{->}(0,0)(0,3)}%
\ThreeDput[normal=0 0 1](1,1,3){\psframe*[linecolor=gray95](2,2)}%
\ThreeDput[normal=0 0 1](1,1,2){\psframe*[linecolor=gray95](2,2)}%
\ThreeDput[normal=0 0 1](1,1,1){\psframe*[linecolor=gray95](2,2)}%
\ThreeDput[normal=0 0 1](0,0,3){\psline[linecolor=gray05](2,1)(2,3)\psline[linecolor=gray05](1,2)(3,2)%
	\rput(2,2){$\times$}}
\ThreeDput[normal=0 0 1](0,0,2){\psline[linecolor=gray05](2,1)(2,3)\psline[linecolor=gray05](1,2)(3,2)%
	\rput(2,1){$\times$}\rput(2,2){$\times$}\rput(2,3){$\times$}\rput(3,2){$\times$}\rput(1,2){$\times$}}%
\ThreeDput[normal=0 0 1](0,0,1){\psline[linecolor=gray05](2,1)(2,3)\psline[linecolor=gray05](1,2)(3,2)%
	\rput(2,2){$\times$}}
\ThreeDput[normal=1 0 0](1,3,0){\uput[0](0,0){\color{blue}$n-1$}}%
\ThreeDput[normal=1 0 0](2,3,0){\uput[0](0,0){\color{blue}$n$}}%
\ThreeDput[normal=1 0 0](3,3,0){\uput[0](0,0){\color{blue}$n+1$}}%
\ThreeDput[normal=0 1 0](3,1,0){\uput[180](0,0){\color{blue}$m-1$}}%
\ThreeDput[normal=0 1 0](3,2,0){\uput[180](0,0){\color{blue}$m$}}%
\ThreeDput[normal=0 1 0](3,3,0){\uput[180](0,0){\color{blue}$m+1$}}%
\ThreeDput[normal=1 0 0](0,3,1){\uput[0](0,0){\color{blue}$p-1$}}%
\ThreeDput[normal=1 0 0](0,3,2){\uput[0](0,0){\color{blue}$p$}}%
\ThreeDput[normal=1 0 0](0,3,3){\uput[0](0,0){\color{blue}$p+1$}}%
\end{pspicture}
\end{tabular}}
\caption{Forward-difference stencil depicting implicit finite-difference scheme \eqref {Eqn:DiffEq2} at site $(m , n , p)$ at time $t _k$. The circles represent known data at the $k$th iteration of the method, while the crosses denote the unknown variables. \label{Fig0-5}}
\end{figure}
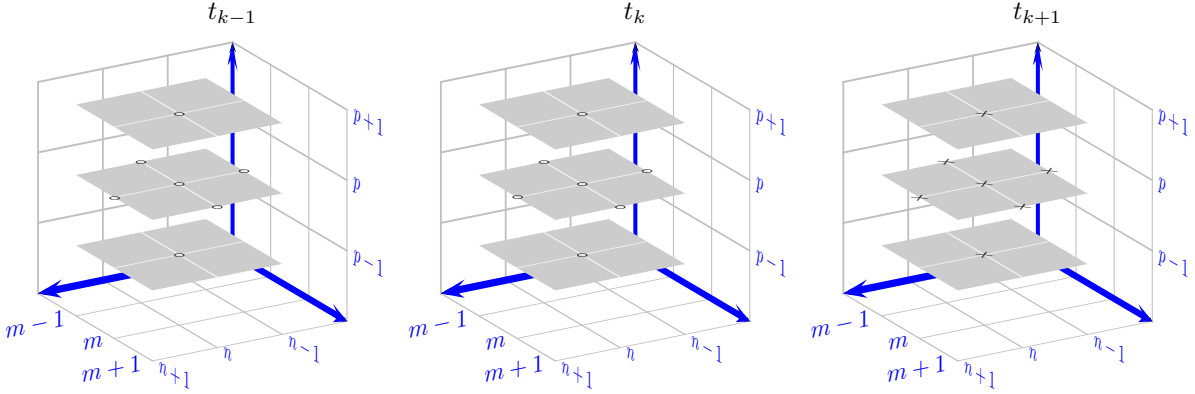

In the following and for the sake of simplification, we will set $N = N _x = N _y = N _z$.

\begin{proposition}
The discrete rate of change of energy with respect to time of system \eqref {Eqn:DiffEq2} subject to the discrete boundary conditions \eqref {Eqn:BoundCond} at the $k$th instant of time is given by \label{Thm:Appendix}
\begin{equation}
\begin{array}{rcl}
\displaystyle {\frac {E ^k - E ^{k - 1}} {\Delta t}} & = & \displaystyle {\left\{ - \sum _{i , j = 1} ^N \left[ \frac {\delta _x u _{0 , i , j} ^k} {(\Delta x) ^2} \frac {\delta _t u _{0 , i , j} ^k} {2 \Delta t} + \frac {\delta _y u _{i , 0 , j} ^k} {(\Delta y) ^2} \frac {\delta _t u _{i , 0 , j} ^k} {2 \Delta t} + \frac {\delta _z u _{i , j , 0} ^k} {(\Delta z) ^2} \frac {\delta _t u _{i , j , 0} ^k} {2 \Delta t} \right] \right.} \\
 & & \quad \displaystyle {- \beta \left\{ \sum _{m , n , p = 1} ^N \left[ \left( \frac {\delta _{t x} u _{m - 1 , n , p} ^k} {2 \Delta x \Delta t} \right) ^2 + \left( \frac {\delta _{t y} u _{m , n - 1 , p} ^k} {2 \Delta y \Delta t} \right) ^2 + \left( \frac {\delta _{t z} u _{m , n , p - 1} ^k} {2 \Delta z \Delta t} \right) ^2 \right] \right.} \\
 & & \qquad \displaystyle {\left. + \sum _{i , j = 1} ^N \left[ \frac {(\delta _{t x} u _{0 , i , j} ^k) (\delta _t u _{0 , i , j} ^k)} {(2 \Delta x \Delta t) ^2} + \frac {(\delta _{t y} u _{i , 0 , j} ^k) (\delta _t u _{i , 0 , j} ^k)} {(2 \Delta y \Delta t) ^2} + \frac {(\delta _{t z} u _{i , j , 0} ^k) (\delta _t u _{i , j , 0} ^k)} {(2 \Delta z \Delta t) ^2}\right] \right\}} \\
 & & \quad\qquad \displaystyle {\left. - \gamma \sum _{m , n , p = 1} ^N \left( \frac {\delta _t u _{m , n , p} ^k} {2 \Delta t} \right) ^2 \right\} \Delta \upsilon .}
\end{array}
\end{equation}
\end{proposition}

\begin{proof}
For the sake of simplification, we adopt the convention $\delta _t ^+ u _{m , n , p} ^k = u _{m , n , p} ^{k + 1} - u _{m , n , p} ^k$ for every $k = 0 , 1 , \dots M - 1$, $m = 0 , 1 , \dots , N _x + 1$, $n = 0 , 1 , \dots , N _y + 1$, $p = 0 , 1 , \dots , N _z + 1$. Notice then that 
\begin{equation*}
\displaystyle {\frac {1} {2} \left( \frac {\delta _t ^+ u _{m , n , p} ^k} {\Delta t} \right) - \frac {1} {2} \left( \frac {\delta _t ^+ u _{m , n , p} ^{k- 1}} {\Delta t} \right)} = \displaystyle {\frac {\delta _t ^2 u _{m , n , p} ^k} {(\Delta t) ^2} \frac {\delta _t u _{m , n , p} ^k} {2}} \end{equation*}
and
\begin{equation*}
\displaystyle {\frac {G (u _{m , n , p} ^{k + 1}) + G (u _{m , n , p} ^k)} {2} - \frac {G (u _{m , n , p} ^k) + G (u _{m , n , p} ^{k - 1})} {2}} = \displaystyle {\frac {G (u _{m , n , p} ^{k + 1}) - G (u _{m , n , p} ^{k - 1})} {u _{m , n , p} ^{k + 1} - u _{m , n , p} ^{k - 1}} \frac {\delta _t u _{m , n , p} ^k} {2}}.
\end{equation*}
Moreover,
\begin{equation}
\begin{array}{c}
\begin{array}{c}
\displaystyle {\frac {(\delta _x u _{m , n , p} ^{k + 1}) (\delta _x u _{m , n , p} ^k)} {2 (\Delta x) ^2} - \frac {(\delta _x u _{m , n , p} ^k) (\delta _x u _{m , n , p} ^{k - 1})} {2 (\Delta x) ^2} = \left( \frac {\delta _x u _{m , n , p} ^k} {\Delta x} \right) \left( \frac {\delta _{t x} u _{m , n , p} ^k} {2 \Delta x} \right)}
\end{array}\\
\begin{array}{rcl}
 \qquad & = & \displaystyle {- \frac {\delta _x ^2 u _{m , n , p} ^k} {(\Delta x) ^2} \frac {\delta _t u _{m , n , p} ^k} {2} + \frac {\delta _x u _{m , n , p} ^k} {(\Delta x) ^2} \frac {\delta _t u _{m + 1 , n , p} ^k} {2} - \frac {\delta _x u _{m - 1 , n , p} ^k} {(\Delta x) ^2} \frac {\delta _t u _{m , n , p} ^k} {2}},
\end{array}
\end{array} \label{Eq:Identity}\tag{*}
\end{equation}
which is an expression that may be further simplified for computational purposes as a consequence of the convention
\begin{equation*}
j _x u _{m , n , p} ^k = \frac {\delta _x u _{m , n , p} ^k} {(\Delta x) ^2} \frac {\delta _t u _{m + 1 , n , p} ^k} {2}.
\end{equation*}
Summing the identities \eqref {Eq:Identity} over all indexes $m$, $n$ and $p$, noticing the presence of a telescoping sum, and applying the discrete Neumann boundary condition, we obtain the following sequence of equalities:
\begin{equation*}
\begin{array}{c}
\begin{array}{c}
\displaystyle {\sum _{m , n , p = 1} ^N \left[\frac {(\delta _x u _{m , n , p} ^{k + 1}) (\delta _x u _{m , n , p} ^k)} {2 (\Delta x) ^2} - \frac {(\delta _x u _{m , n , p} ^k) (\delta _x u _{m , n , p} ^{k - 1})} {2 (\Delta x) ^2} \right] = \qquad \qquad \qquad}
\end{array}\\
\begin{array}{rcl}
 \qquad\qquad & = & \displaystyle {- \sum _{m , n , p = 1} ^N \frac {\delta _x ^2 u _{m , n , p} ^k} {(\Delta x) ^2} \frac {\delta _t u _{m , n , p} ^k} {2} + \sum _{n , p = 1} ^N \left[ \sum _{m = 1} ^N \left( j _x u _{m , n , p} ^k - j _x u _{m - 1 , n , p} ^k \right) \right]} \\
 & = & \displaystyle {- \sum _{m , n , p = 1} ^N \frac {\delta _x ^2 u _{m , n , p} ^k} {(\Delta x) ^2} \frac {\delta _t u _{m , n , p} ^k} {2} - \sum _{n , p = 1} ^N \frac {\delta _x u _{0 , n , p} ^k} {(\Delta x) ^2} \frac {\delta _t u _{1 , n , p} ^k} {2}. }
\end{array}
\end{array} 
\end{equation*}
In similar fashion, one can verify that substituting the difference $\delta _x$ for $\delta _y$ or $\delta _z$ in both sides of this last equation yields a valid equality. It is now straight-forward to verify that
\begin{eqnarray*}
\sum _{m , n , p = 1} ^N \frac {\delta _t ^+ H _{m , n , p} ^{k - 1}} {\Delta t} & = & \sum _{m , n , p = 1} ^N \frac {\delta _t u _{m , n , p} ^k} {2 \Delta t} \left[ \frac {\delta _t ^2 u _{m , n , p} ^k} {(\Delta t) ^2} - \left(\frac {\delta _x ^2 u _{m , n , p} ^k} {(\Delta x) ^2} + \frac {\delta _y ^2 u _{m , n , p} ^k} {(\Delta y) ^2} + \frac {\delta _z ^2 u _{m , n , p} ^k} {(\Delta z) ^2} \right) \right.\\
 & & \quad + \left. \frac {G (u _{m , n , p} ^{k + 1}) - G (u _{m , n , p} ^ {k - 1})} {u _{m , n , p} ^{k + 1} - u _{m , n , p} ^{k - 1}}\right] - \sum _{i , j = 1} ^N \left[ \frac {\delta _x u _{0 , i , j} ^k} {(\Delta x) ^2} \frac {\delta _t u _{1 , i , j} ^k} {2 \Delta t} \right. \\
 & & \qquad \left. + \frac {\delta _y u _{i , 0 , j} ^k} {(\Delta y) ^2} \frac {\delta _t u _{i , 1 , j} ^k} {2 \Delta t} + \frac {\delta _z u _{i , j , 0} ^k} {(\Delta z) ^2} \frac {\delta _t u _{i , j , 1} ^k} {2 \Delta t} \right]\\
 & = & \beta \sum _{m , n , p = 1} ^N \frac {\delta _t u _{m , n , p} ^k} {(2 \Delta t) ^2} \left[ \frac {\delta _t \delta _x ^2 u _{m , n , p} ^k} {(\Delta x) ^2} + \frac {\delta _t \delta _y ^2 u _{m , n , p} ^k} {(\Delta y) ^2} + \frac {\delta _t \delta _z ^2 u _{m , n , p} ^k} {(\Delta z) ^2} \right]\\
 & & \quad - \gamma \sum _{m , n , p} ^N \left( \frac {\delta _t u _{m , n , p} ^k} {2 \Delta t} \right) ^2 - \sum _{i , j = 1} ^N \left[ \frac {\delta _x u _{0 , i , j} ^k} {(\Delta x) ^2} \frac {\delta _t u _{1 , i , j} ^k} {2 \Delta t} \right. \\
 & & \qquad \left. + \frac {\delta _y u _{i , 0 , j} ^k} {(\Delta y) ^2} \frac {\delta _t u _{i , 1 , j} ^k} {2 \Delta t} + \frac {\delta _z u _{i , j , 0} ^k} {(\Delta z) ^2} \frac {\delta _t u _{i , j , 1} ^k} {2 \Delta t} \right].
\end{eqnarray*}
Commutativity of the discrete operators yields $\delta _t u _{m , n , p} ^k \delta _t \delta _x ^2 u _{m , n , p} ^k = \delta _t u _{m , n , p} ^k \left(\delta _t u _{m + 1 , n , p} ^k - 2 \delta _t u _{m , n , p} ^k + u _{m - 1 , n , p} ^k\right)$. 
Summing over all indexes $m = 1 , 2 , \dots , N$, applying the discrete version of Green's first identity, and summing next over all indexes $n , p = 1 , 2 , \dots , N$, it follows that
\begin{equation*}
\sum _{m , n , p = 1} ^N \delta _t u _{m , n , p} ^k \delta _t \delta _x ^2 u _{m , n , p} ^k = - \sum _{n , p = 1} ^N \delta _{t x} u _{0 , n , p} ^k \delta _t u _{0 , n , p} ^k - \sum _{m , n , p = 1} ^N \left( \delta _{t x} u _{m - 1 , n , p} ^k \right) ^2;
\end{equation*}
moreover, similar relations may be obtained for the cases when the operator $\delta _x$ is replaced by $\delta _y$ or $\delta _z$. In this circumstances, we readily obtain that 
\begin{eqnarray*}
\frac {\delta _t ^+ E ^{k - 1}} {\Delta t} & = & - \beta \left\{ \sum _{m , n , p = 1} ^N \left[ \left( \frac {\delta _{t x} u _{m - 1 , n , p} ^k} {2 \Delta x \Delta t} \right) ^2 + \left( \frac {\delta _{t y} u _{m , n - 1 , p} ^k} {2 \Delta y \Delta t} \right) ^2 + \left( \frac {\delta _{t z} u _{m , n , p - 1} ^k} {2 \Delta z \Delta t} \right) ^2 \right] \Delta \upsilon \right. \\
 & & \quad \displaystyle {\left. + \sum _{i , j = 1} ^N \left[ \frac {(\delta _{t x} u _{0 , i , j} ^k) (\delta _t u _{0 , i , j} ^k)} {(2 \Delta x \Delta t) ^2} + \frac {(\delta _{t y} u _{i , 0 , j} ^k) (\delta _t u _{i , 0 , j} ^k)} {(2 \Delta y \Delta t) ^2} + \frac {(\delta _{t z} u _{i , j , 0} ^k) (\delta _t u _{i , j , 0} ^k)} {(2 \Delta z \Delta t) ^2}\right] \right\} \Delta \upsilon} \\
 & & \qquad - \gamma \sum _{m , n , p = 1} ^N \left( \frac {\delta _t u _{m , n , p} ^k} {2 \Delta t} \right) ^2 \Delta \upsilon - \sum _{i , j = 1} ^N \left[ \frac {\delta _x u _{0 , i , j} ^k} {(\Delta x) ^2} \frac {\delta _t u _{1 , i , j} ^k} {2 \Delta t} + \frac {\delta _y u _{i , 0 , j} ^k} {(\Delta y) ^2} \frac {\delta _t u _{i , 1 , j} ^k} {2 \Delta t} \right. \\
 & & \quad\qquad \left.  + \frac {\delta _z u _{i , j , 0} ^k} {(\Delta z) ^2} \frac {\delta _t u _{i , j , 1} ^k} {2 \Delta t} \right] \Delta \upsilon + \sum _{i , j = 1} ^N \left[ \frac {(\delta _{t x} u _{0 , i , j} ^k) (\delta _x u _{0 , i , j} ^k)} {2 (\Delta x) ^2 \Delta t} + \frac {(\delta _{t y} u _{i , 0 , j} ^k) (\delta _y u _{i , 0 , j} ^k)} {2 (\Delta y) ^2 \Delta t} \right. \\
 & & \qquad\qquad \displaystyle {\left. + \frac {(\delta _{t z} u _{i , j , 0} ^k) (\delta _z u _{i , j , 0} ^k)} {2 (\Delta z) ^2 \Delta t} \right] \Delta \upsilon,}
\end{eqnarray*}
whence the result follows after an easy simplification in the last two sums.
\end{proof}

A direct comparison between this result and Propositions \ref {Prop1} and \ref {Prop:2-4} shows that the numerical method presented here consistently approximates the derivative of the total energy of a problem described by either \eqref {Eqn:Main} or \eqref {Eqn:DiscreteMain}. As a consequence, the method proposed in this work is capable of preserving the total energy of a conservative system described by the continuous equation \eqref {Eqn:Main} or the discrete system \eqref {Eqn:DiscreteMain}, which is a physical scenario that appears when both $\beta$ and $\gamma$ are equal to zero and when a fixed boundary is considered. 

For the next result, \emph {stability} means stability order $n$ (see \cite {Thomas}).

\begin{proposition}
Let $V ^\prime$ be identically equal to zero, and let $J = 0$. In order for scheme \eqref {Eqn:DiffEq2} to be stable it is necessary that the condition 
\begin{equation}
4 \left( \frac {1} {(\Delta x) ^2} + \frac {1} {(\Delta y) ^2} + \frac {1} {(\Delta z ) ^2} \right) [(\Delta t) ^2 - \beta \Delta t] - \left[ \gamma + \mathfrak {m} ^2 \Delta t \right] \Delta t < 4
\end{equation} 
be satisfied. \label{Prop:3-2}
\end{proposition}

\begin{proof}
For every $m , n , p = 1 , \dots , N$ and every $k = 0 , \dots , M - 1$, let $\bar {u} _{1 , m , n , p} ^{k + 1} = u _{m , n , p} ^{k + 1}$ and $\bar {u} _{2, m , n , p} ^{k + 1} = u _{m , n , p} ^k$, and let $\mathbf {u} _{m , n , p} ^k$ be the two-dimensional column vector whose components are $\bar {u} _{1 , m , n , p} ^k$ and $\bar {u} _{2 , m , n , p} ^k$. In these terms, scheme \eqref {Eqn:DiffEq2} can be presented as
\begin{equation*}
\left(\begin{array}{cc}
g & 0 \\
0 & 1
\end{array}\right) \left( \mathbf {u} _{m , n , p} ^{k + 1} \right) _{m , n , p = 1} ^N =
\left( \begin{array}{cc}
2 + (\Delta t) ^2 \left[\frac {\delta _x ^2} {(\Delta x) ^2} + \frac {\delta _y ^2} {(\Delta y) ^2} + \frac {\delta _z ^2} {(\Delta z) ^2}\right] & - h \\
1 & 0
\end{array} \right) \left( \mathbf {u} _{m , n , p} ^k \right) _{m , n , p = 1} ^N,
\end{equation*}
where
\begin{eqnarray*}
g & = & 1 - \frac {\beta \Delta t} {2} \left[\frac {\delta _x ^2} {(\Delta x) ^2} + \frac {\delta _y ^2} {(\Delta y) ^2} + \frac {\delta _z ^2} {(\Delta z) ^2}\right] + \frac {\gamma \Delta t} {2} + \frac {\mathfrak {m} ^2 (\Delta t) ^2} {2}, \\
h & = & 1 + \frac {\beta \Delta t} {2} \left[\frac {\delta _x ^2} {(\Delta x) ^2} + \frac {\delta _y ^2} {(\Delta y) ^2} + \frac {\delta _z ^2} {(\Delta z) ^2}\right] - \frac {\gamma \Delta t} {2} + \frac {\mathfrak {m} ^2 (\Delta t) ^2} {2}.
\end{eqnarray*}
We apply Fourier transform in order to reach the expression
\begin{equation*}
\hat {\mathbf {u}} _{m , n , p} ^{k + 1} = 
\left( \begin{array}{cc}
\frac {2} {\hat {g} ( \xi , \zeta , \varsigma )} \left(1 - 2 (\Delta t) ^2 \left( \frac {\sin ^2 \frac {\xi} {2}} {(\Delta x) ^2} + \frac {\sin ^2 \frac {\zeta} {2}} {(\Delta y) ^2} + \frac {\sin ^2 \frac {\varsigma} {2}} {(\Delta z) ^2} \right) \right) & - \frac {\hat {h} ( \xi , \zeta , \varsigma )} {\hat {g} ( \xi , \zeta , \varsigma )} \\
1 & 0
\end{array} \right) 
\hat {\mathbf {u}} _{m , n , p} ^k,
\end{equation*}
where the `hat' operator obviously denotes Fourier transform. We identify the $2 \times 2$ matrix multiplying $\hat {\mathbf {u}} _{m , n , p} ^k$ in the above equation as the amplification matrix $A (\xi , \zeta , \varsigma)$ of our problem. Moreover, it is easy to check that the eigenvalues of this matrix when $\xi$, $\zeta$ and $\varsigma$ are all equal to $\pi$, are given by
\begin{equation*}
\lambda _\pm = \frac {1 - 2 R ^2 (\Delta t) ^2 \pm \sqrt{ (1 - 2 R ^2 (\Delta t) ^2) ^2 - \hat {h} (\pi , \pi , \pi) \hat {g} (\pi , \pi , \pi)}} {\hat {g} (\pi , \pi , \pi)},
\end{equation*}
where 
\begin{equation*}
R ^2 = \frac {1} {(\Delta x) ^2} + \frac {1} {(\Delta y) ^2} + \frac {1} {(\Delta z ) ^2}.
\end{equation*}

Suppose for a moment that $1 - 2 R ^2 (\Delta t) ^2 < - \hat {g} (\pi , \pi , \pi)$. If the radical in the expression above yields a pure real number then $| \lambda _- | > 1$. So for every positive integer $l$, $|| A ^l || \geq | \lambda _- | ^l$ grows faster than $K _1 + l K _2$ for any constants $K _1$ and $K _2$. A similar situation prevails when the radical is a pure imaginary number, except that in this case $| \cdot |$ represents the usual Euclidean norm in the field of complex numbers. Therefore in order for our numerical method to be stable it is necessary that $1 - 2 R ^2 (\Delta t) ^2 > - \hat {g} (\pi , \pi , \pi)$, which is what we wished to establish.
\end{proof}

\begin{corollary}
Let $V ^\prime$ be identically equal to zero, let $J = 0$, and suppose that $\Delta x = \Delta y = \Delta z$. In order for scheme \eqref {Eqn:DiffEq2} to be stable it is necessary that the condition 
\begin{equation}
(12 R ^2 - \mathfrak {m} ^2) (\Delta t) ^2 - (\gamma + 12 \beta R ^2) \Delta t < 4
\end{equation}
be satisfied, for $R = 1 / \Delta x$. \qed
\end{corollary}

It is important to notice that the order of consistency of the method as defined by the truncation error is $\mathcal {O} ((\Delta x) ^2 + (\Delta y) ^2 + (\Delta z) ^2 + (\Delta t) ^2)$. Also, it is worth mentioning that we have approximated $V ^\prime (u (m \Delta x , n \Delta y , p \Delta z , k \Delta t))$ in the right-hand side of Eq. \eqref {Eqn:Main} through the discrete derivative of $V$ with respect to $u$ presented in \eqref {Eqn:DiffEq2} and not through the direct evaluation $V ^\prime (u _{m , n , p} ^k)$ in view of the fact that such standard scheme is known to be highly unstable \cite {StraussVazquez}. 

\begin{figure}
\centerline{%
\begin{tabular}{cc}
\includegraphics[width=0.47\textwidth]{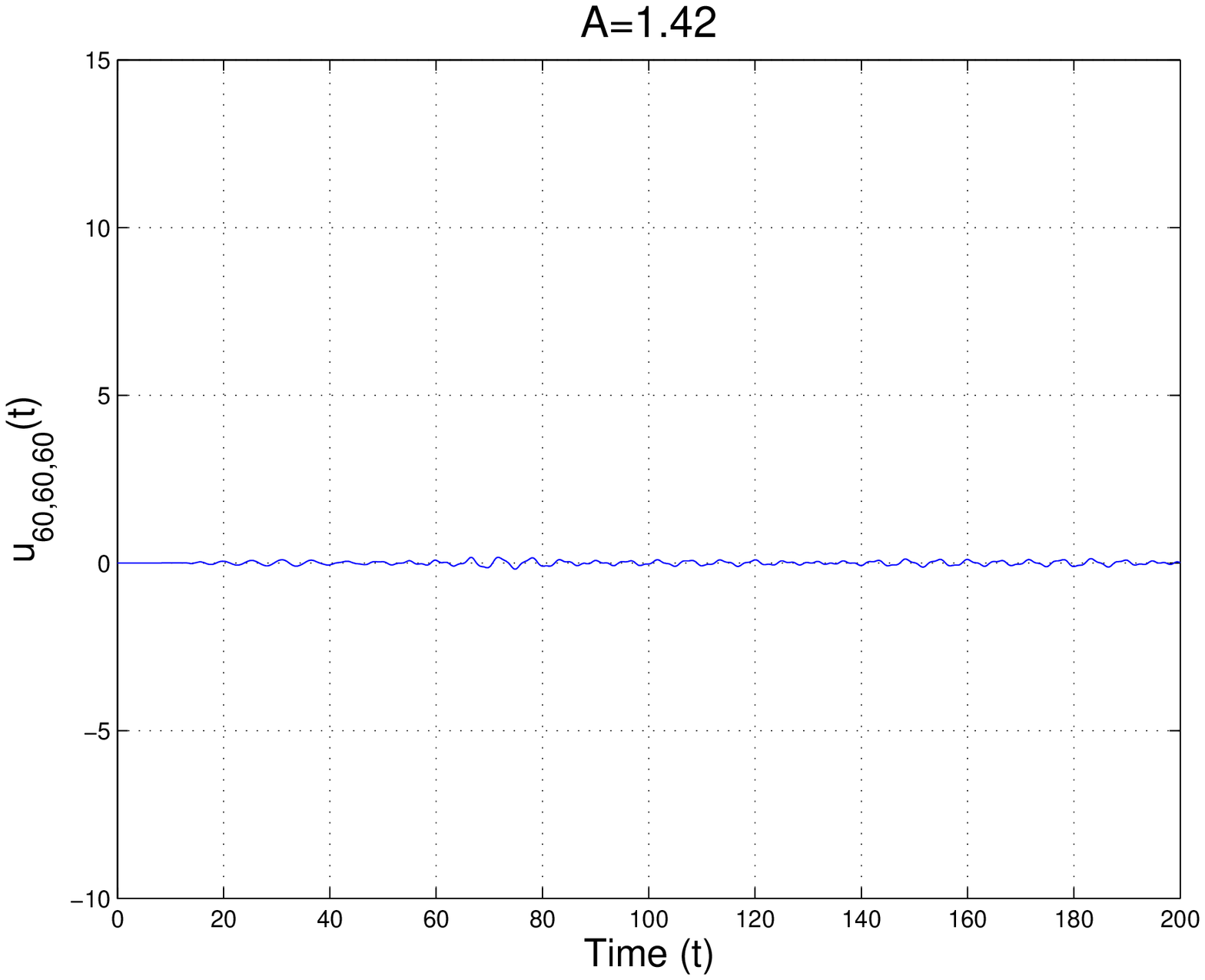}&\includegraphics[width=0.47\textwidth]{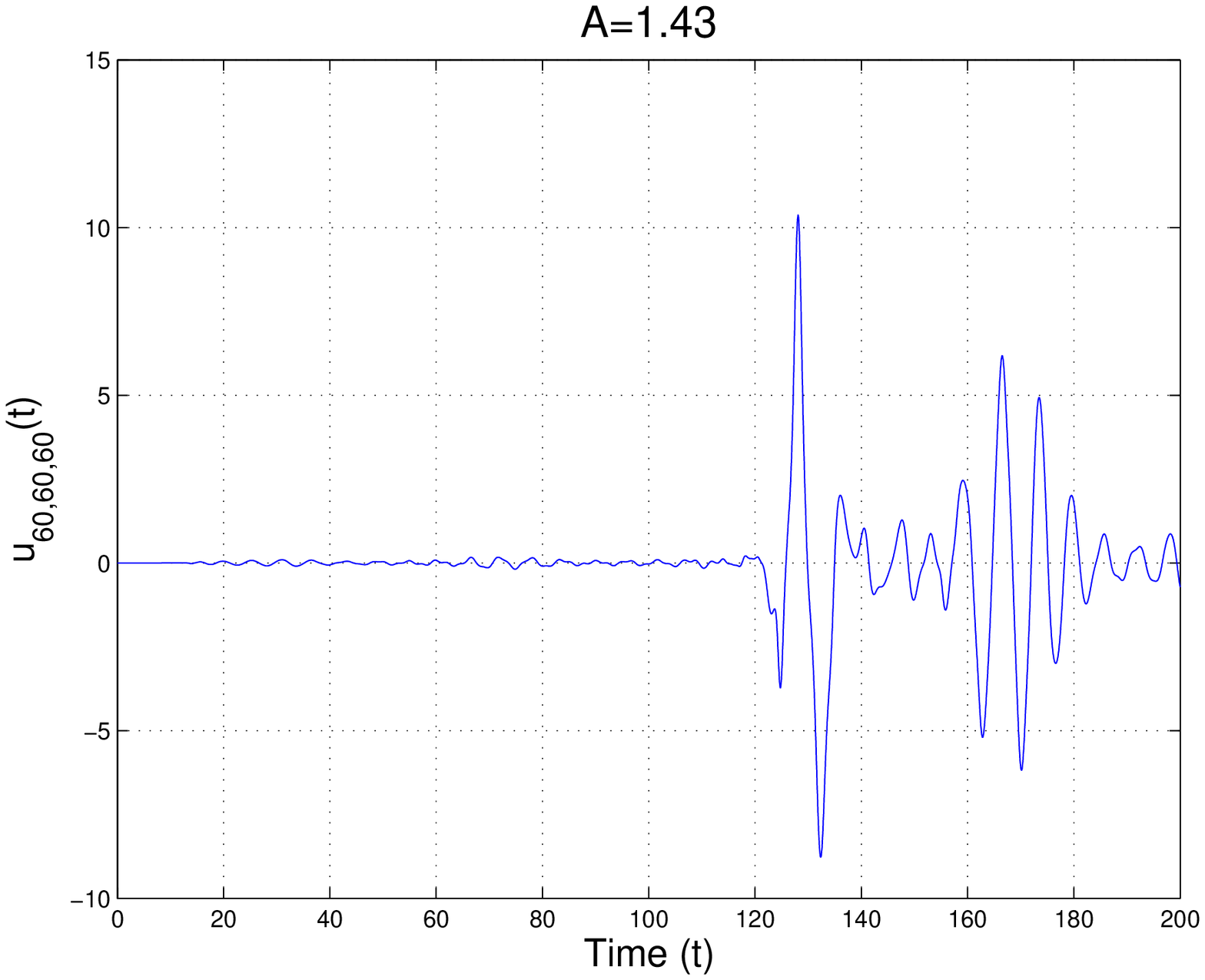}
\end{tabular}}
\caption{Graph of the approximate solution $u _{60 , 60 , 60}$ versus time, obtained by driving system \eqref {Eqn:DiscreteMain} at a frequency equal to $0.9$ in the forbidden band-gap of the continuous-limit medium, and two different amplitudes: $1.42$ (left) and $1.43$ (right). All other parameters are set equal to zero, a time step equal to $0.05$ was fixed, and a system with an absorbing boundary consisting of $200 ^3$ coupled nodes was considered. The graphs are presented as evidence of the presence of supratransmission in the system under study. \label{Fig1}}
\end{figure}

\subsection{Spherical problem}

In order to approximate solutions to \eqref {Eqn:MainRad} in the closure of the solid sphere with center in the origin and radius $L$, we let $\epsilon = r_0 < r _1 < \dots < r _{M + 1} = L$ and $0 = t _0 < t _1 < \dots < t _N = T$ be regular partitions of $[\epsilon , L]$ and $[0 , T]$, respectively, into $M + 1$ and $N$ subintervals of lengths $\Delta r$ and $\Delta t$, respectively. Denote the approximate value of $v (r _j , t _k)$ by $v _j ^k$. The finite-difference scheme used to pursue that task is the implicit method
\begin{eqnarray}
\frac {v _j ^{k + 1} - 2 v _j ^k + v _j ^{k - 1}} {(\Delta t) ^2} - \frac {v _{j + 1} ^k - 2 v _j ^k + v _{j - 1} ^k} {(\Delta r) ^2} + \gamma \frac {v _j ^{k + 1} - v _j ^{k - 1}} {2 \Delta t} - \qquad & & \nonumber \\
\beta \frac {\left( v _{j + 1} ^{k + 1} - 2 v _j ^{k + 1} + v _{j - 1} ^{k + 1}\right) - \left( v _{j + 1} ^{k - 1} - 2 v _j ^{k - 1} + v _{j - 1} ^{k - 1} \right)} {2 \Delta t \left( \Delta r \right) ^2} + \frac {\mathfrak {m} ^2} {2} \left[ v _j ^{k + 1} + v _j ^{k - 1} \right] + \quad & & \label{EasyScheme1} \\
 (\epsilon + j \Delta r) ^2 \frac {V ( \frac {v _ j ^{k + 1}} {\epsilon + j \Delta r} ) - V ( \frac {v _j ^{k - 1}} {\epsilon + j \Delta r} )} {v _j ^{k + 1} - v _j ^{k - 1}} - J \left( \epsilon + j \Delta r \right) & = & 0, \nonumber
\end{eqnarray}
defined for every $j = 1 , \dots , M$ and every $k = 1 , \dots , N - 1$. Meanwhile, the total energy of the system \eqref {Eqn:RadEnergy} and the Neumann boundary condition of the problem \eqref {Eqn:BoundRad} at the $k$-th time step will be respectively approximated using the schemes
\begin{equation}
\begin{array}{rcl}
\displaystyle {E ^k} & = & \displaystyle {\frac {1} {2} \sum _{j = 0} ^{M - 1} \left( \frac {v _j ^{k + 1} - v _j ^k } {\Delta t} \right) ^2 \Delta r + \frac {1} {2} \sum _{j = 0} ^{M - 1} \left( \frac {v _{j + 1} ^{k + 1} - v _j ^{k + 1}} {\Delta r}\right) \left( \frac {v _{j + 1} ^k - v _j ^k} {\Delta r} \right) \Delta r} \\ %
 & & \displaystyle {\quad + \frac {\mathfrak {m} ^2} {2} \sum _{j = 0} ^{M - 1} \frac {(v _j ^{k + 1}) ^2 +  (v _j ^k) ^2} {2} \Delta r + \sum _{j = 1} ^{M - 1} (\epsilon + j \Delta r) ^2 \frac {V \left( \frac {v _j ^{k + 1}} {\epsilon + j \Delta r} \right) + V \left( \frac {v _j ^k} {\epsilon + j \Delta r} \right)} {2} \Delta r} \\
 & & \displaystyle {\qquad - J \sum _{j = 0} ^{M - 1} (\epsilon + j \Delta r) \frac {v _j ^{k + 1} + v _j ^k} {2} \Delta r,}
\end{array}
\end{equation}
and
\begin{equation}
\frac {v _{M + 1} ^k - v _M ^k} {\Delta r} + \frac {v _{M + 1} ^k + v _M ^k} {2 (\epsilon + M \Delta r) ^2} = 0. \label {Eqn:BoundCondRad}
\end{equation}

The following results summarize the most important numerical properties of the method proposed in the present subsection. The proofs are omitted in view that they are similar to the corresponding results of the Cartesian case.

\begin{proposition}[Mac\'{\i}as-D\'{\i}az et al \cite {Macias-Puri}]
The discrete rate of change of energy with respect to time of system \eqref {EasyScheme1} subject to the discrete boundary condition \eqref {Eqn:BoundCondRad} at the $k$th instant of time is given by 
\begin{equation}
\begin{array}{rcl}
\displaystyle {\frac {E ^k - E ^{k - 1}} {\Delta t}} & = & \displaystyle {- \frac {\pi} {2} \left\{ \beta \sum _{j = 1} ^{M - 1} \left( \frac { v _j ^{k + 1} - v _j ^{k - 1} } {2 \Delta t} \right) \left( \frac { ( v _j ^{k + 1} - v _j ^{k - 1} ) - ( v _{j - 1} ^{k + 1} - v _{j - 1} ^{k - 1} ) } { \Delta t ( \Delta r ) ^2} \right) \Delta r \right.} \\
 & & \displaystyle {\qquad \left. + \gamma \sum _{j = 1} ^{M - 1} \left( \frac {v _j ^{k + 1} - v _j ^{k - 1}} {2 \Delta t} \right) ^2 \Delta r \right\},}
\end{array}
\end{equation} 
for $\epsilon$ equal to zero. \qed
\end{proposition}

\begin{proposition}[Mac\'{\i}as-D\'{\i}az et al \cite {Macias-Puri}]
Let $V ^\prime$ be identically equal to zero, and let $J = 0$. In order for finite-difference scheme \eqref {Eqn:MainRad} to be stable
order $n$ it is necessary that
\begin{equation}
\left( \frac {\Delta t} {\Delta r} \right) ^2 < 1 + \gamma \frac {\Delta t} {4} + \beta \frac {\Delta t} { \left( \Delta r \right) ^2 } + \mathfrak {m} ^2 \frac {(\Delta t) ^2} {4}.
\end{equation} \qed
\end{proposition}

We must remark that the computational technique presented in this section makes use of Newton's method to approximate solutions of systems of nonlinear equations. In each iteration, the system of equations derived from Newton's method is linear, and may be solved using Crout's technique for tridiagonal systems.

\section{Applications\label{Sec4}}

\begin{figure}
\centerline{\includegraphics[width=0.6\textwidth]{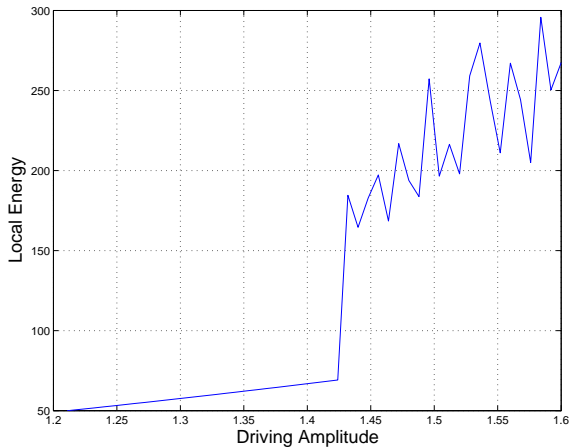}}
\caption{Bifurcation diagram of total energy of node located at sine $(60 , 60 , 60)$ over a time period of $200$ versus driving amplitude, for a diving frequency equal to $0.9$ in the forbidden band-gap region of the continuous-limit medium. All other parameters are equal to zero, a time step equal to $0.05$ was employed, and a system with an absorbing boundary consisting of $200 ^3$ coupled nodes was considered. The graph is presented as evidence of the presence of supratransmission in the system under study. \label{Fig2}}
\end{figure}

Throughout this section, we suppose that the function driving the boundary assumes the expression $\phi (t) = A \sin (\Omega t)$, where the driving frequency takes on values in the forbidden band-gap region of the continuous-limit medium $\Omega < \sqrt {\mathfrak {m} ^2 + 1}$. The explicit scheme resulting from setting $\beta$ equal to zero in \eqref {Eqn:DiffEq2} is employed to verify the validity of our results in the case of weak internal damping.

\subsection{Nonlinear supratransmission}
\label{SubSec:Supra}

The process of supratransmission in nonlinear systems submitted to harmonic driving is completely characterized by a sudden increase in the energy injected into the system by the driving source \cite {Geniet-Leon}. Thus, the method employed to determine the critical value at which supratransmission starts, given a fixed frequency $\Omega$ in the forbidden band-gap of the system, consists in computing the associated energy $E$ for various driving amplitudes $A$ in an interval containing the supratransmission threshold, over a fixed, relatively long period of time; in these circumstances, the graph of $E$ versus $A$ will evidence a point of discontinuity where a drastic increase in the total energy of the system takes place. 

In the case of an infinite number of coupled junctions $(u _{m , n , p}) _{m , n , p = 1} ^\infty$ satisfying problem (\ref {Eqn:DiscreteMain}), we consider a large finite subsystem consisting of $N \times N \times N$ coupled junctions, with damping coefficient $\gamma _{m , n , p}$ including the effect of an absorbing boundary in the farthest junctions from the point of intersection of the three driving boundaries. That is, we let
\begin{equation}
\begin{array}{rcl}
\gamma _{m , n , p} & = & \displaystyle {\gamma + \frac {1} {6} \left[ 3 + \tanh \left( \displaystyle {\frac {2 m - N _0 + N} {6}} \right) + \tanh \left( \displaystyle {\frac {2 n - N _0 + N} {6}} \right) \right.} \\
 & & \quad \displaystyle {\left. + \tanh \left( \frac {2 p - N _0 + N} {6} \right) \right]},
\end{array} \label{Eq:Damp}
\end{equation}
where $1 \ll N _0 < N$. 

Let all constant parameters in the differential equations of \eqref {Eqn:DiscreteMain} be set equal to zero. Following the method described in the previous paragraph, we submit system \eqref {Eqn:DiscreteMain} to harmonic driving with frequency in the forbidden band-gap of the continuous-limit medium over a time period of $[ 0 , 200 ]$, and compute the associated total energy of the system for several driving amplitudes. In order to avoid the generation of shock waves at the origin around the time $t = 0$, we opt for slowly and linearly increase the driving amplitude from $0$ to its actual value $A$. Numerically, we choose a time step equal to $0.05$ (so that the stability condition provided by Proposition \ref {Prop:3-2} is clearly satisfied), fix a cubic system of dimension $200 ^3$, and we damp the farthest nodes from the origin using \eqref {Eq:Damp} and $N _0 = 50$. 

To start with, fix a driving frequency of $0.9$ and chose two different amplitude values: $A = 1.42$ and $A = 1.43$. The time behavior of the solution of the node located at site $(60 , 60 , 60)$ as a result of driving system \eqref {Eqn:DiscreteMain} under the circumstances described in the paragraph above for the two amplitudes considered is displayed in Fig. \ref {Fig1}. It is worth noticing that the wave signals transmitted into the system for a driving amplitude equal to $1.42$ posses a very low amplitude when compared against the driving amplitude itself (left graph). On the other hand, a driving amplitude of $1.43$ produces wave signals of higher amplitude (right graph). Moreover, the total energy at site $(60 , 60 , 60)$ is computed by integrating the discrete Hamiltonian, obtaining, in the first case, an energy equal to $68.7613$, while a total energy of $161.3648$ is obtained in the second, whence the existence of a critical amplitude between $1.42$ and $1.43$ at which supratransmission starts is suspected.


Next, we compute the total energy at site $(60 , 60 , 60)$ of system \eqref {Eqn:DiscreteMain} for several amplitude values around the suspected critical values, and for a fixed frequency equal to $0.9$. As before, all other parameters are set equal to zero, and we use the same numerical setting as before. In these circumstances, we present the graph of total energy versus amplitude on the time period $[ 0 , 200 ]$. The results are presented in Fig. \ref {Fig2} and confirm that a drastic increase in the total energy of the node appears for a driving amplitude between $1.42$ and $1.43$, proving thus the presence of nonlinear supratrasmission in our system, at least for a driving frequency of $0.9$. In this point it must be mentioned that we have established that supratransmission is likewise present for other choices we made of the parameter $\Omega$. 

\subsection{A counter-example}
\label{SubSec:Counter}

Numerically, let us fix a time step $\Delta t = 0.02$, and let the radial step $\Delta r$ and the parameter $\epsilon$ both equal $\Delta t$. In all our computations we consider a fixed time period. Moreover, in order to avoid the generation of shock waves, the driving amplitude will increase linearly and slowly from $0$ to its actual value $A$ during a relatively short period of time, before the initial instant $t = 0$ takes place.

Assume that the medium has no damping. In order to simulate an unbounded medium, we approximate solutions to problem (\ref {Eqn:MainRad}) in a closed sphere $S$ with center in the origin and radius $L = 6$, in which the parameter $\gamma$ slowly increases in magnitude from $0$ to $1$ outside the open sphere with center in the origin and radius $5$, simulating thus an absorbing boundary. More precisely, we let
\begin{equation}
\gamma (r) = \left\{ \begin{array}{ll}
\displaystyle {\frac {1} {2} \left\{ 1 + \tanh [8 (r - 5.5)] \right\}}, & 5 \leq r \leq 6, \\
0, & 0 < r < 5.
\end{array}\right.
\end{equation}

It is worth noticing that if the potential function $V$ is that for a sine-Gordon system, then $r V ^\prime (v / r)$ is approximately equal to zero for nonzero values of $r$ sufficiently close to zero. It is therefore expected that the medium behaves in a linear fashion around the origin and, particularly, that the medium does not support the process of supratransmission under the presence of harmonic perturbations at the origin. Of course, this claim will be confirmed numerically next for both Klein-Gordon and sine-Gordon systems.

\begin{figure}
\centerline{%
\begin{tabular}{cc}
{\bf (a)} & {\bf (b)} \\
\includegraphics[width=0.5\textwidth]{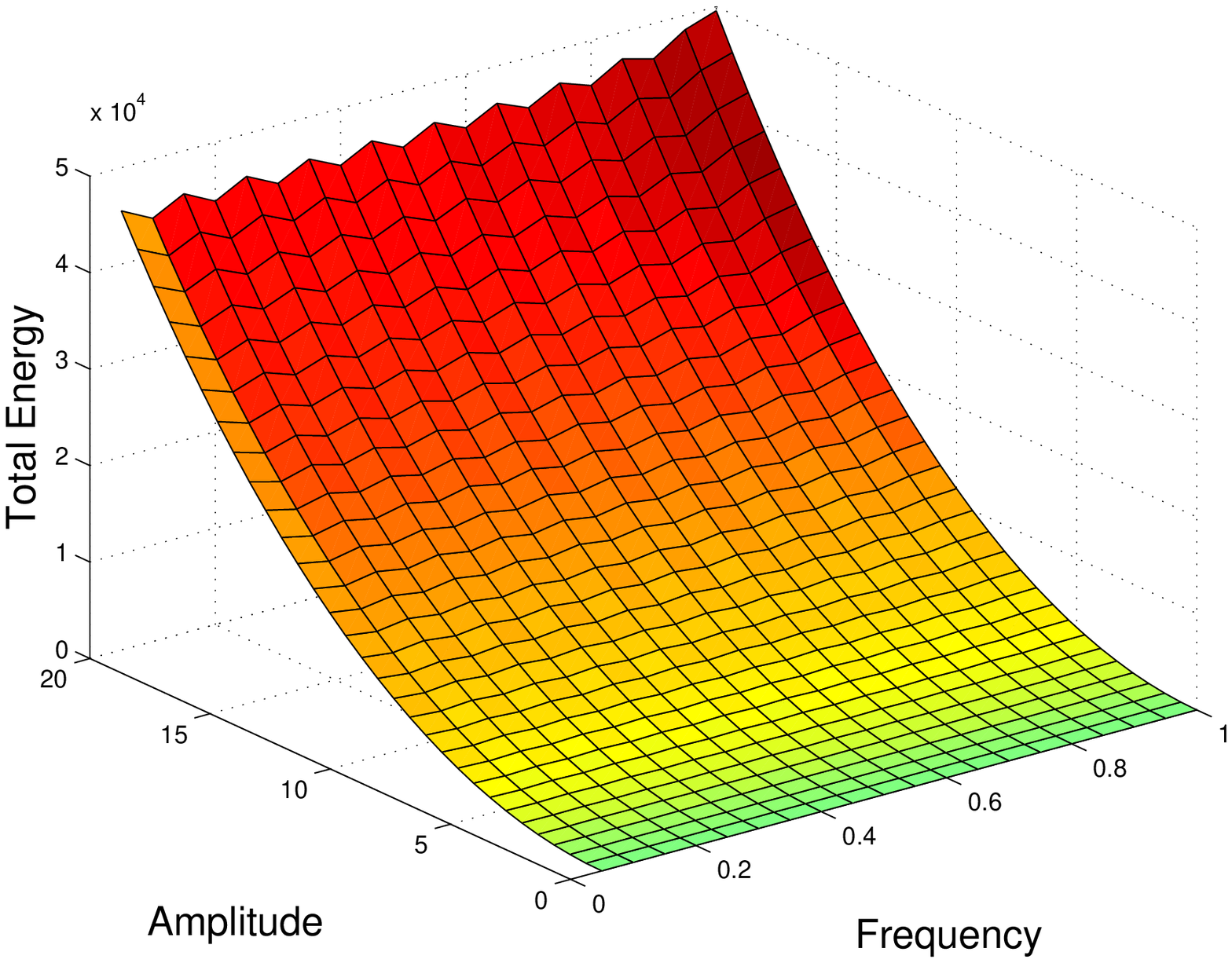}&\includegraphics[width=0.5\textwidth]{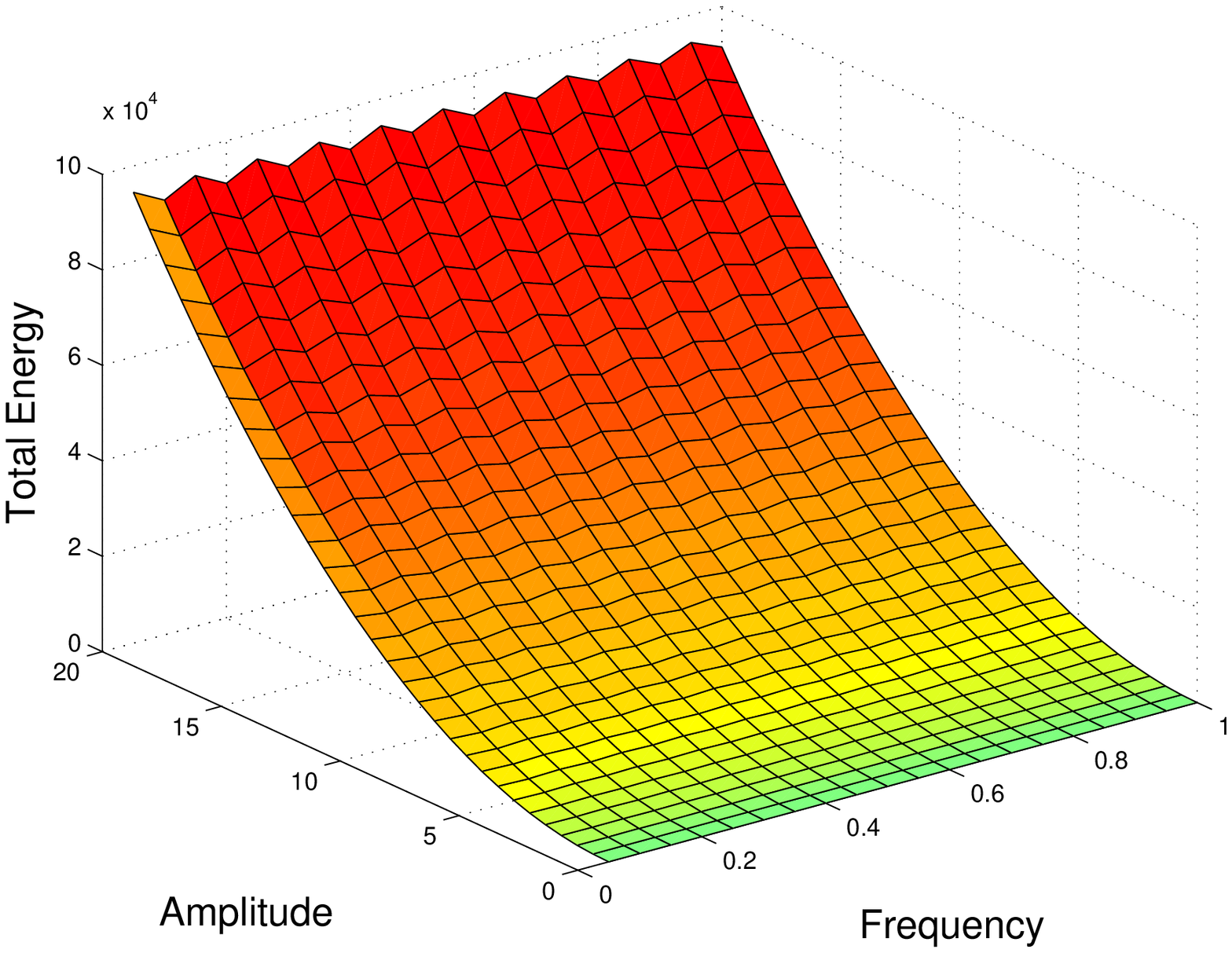}
\end{tabular}}
\caption{Total energy versus driving frequency $\Omega$ and driving amplitude $A$ in undamped system \eqref {Eqn:MainRad} with $J = 0$ and potential equal to that of a nonlinear Klein-Gordon equation (left), and equal to that of a classical sine-Gordon medium (right). The harmonic driving function at the origin is defined by $\phi (t) = A \sin (\Omega t)$, and a time period of $200$ was fixed. \label{Fig12}}
\end{figure}

Throughout, we let $A$ range in $[0 , 20]$. Let us take a fixed frequency $\Omega = 0.9$ in the forbidden band-gap of a continuous Klein-Gordon medium described by model \eqref {Eqn:MainRad}. The associated total energy of the system during the fixed period of time is computed, obtaining 
evidence of a continuous increase in the amount of energy injected in the system by the driving boundary, with no apparent discontinuities in the total energy of the system. Next, we let $\Omega$ take on values in the interval $[0 , 1]$. The graph of total energy versus driving amplitude and driving frequency is presented as Fig. \ref {Fig12}(a), and the results evidence that the total energy increases smoothly as the driving amplitude is increased. This fact supports our claim that the process of nonlinear supratransmission is not present in this medium.

\begin{figure}
\centerline{%
\begin{tabular}{cc}
$t = 5$ & $t = 7.5$\\
\includegraphics[width=0.35\textwidth]{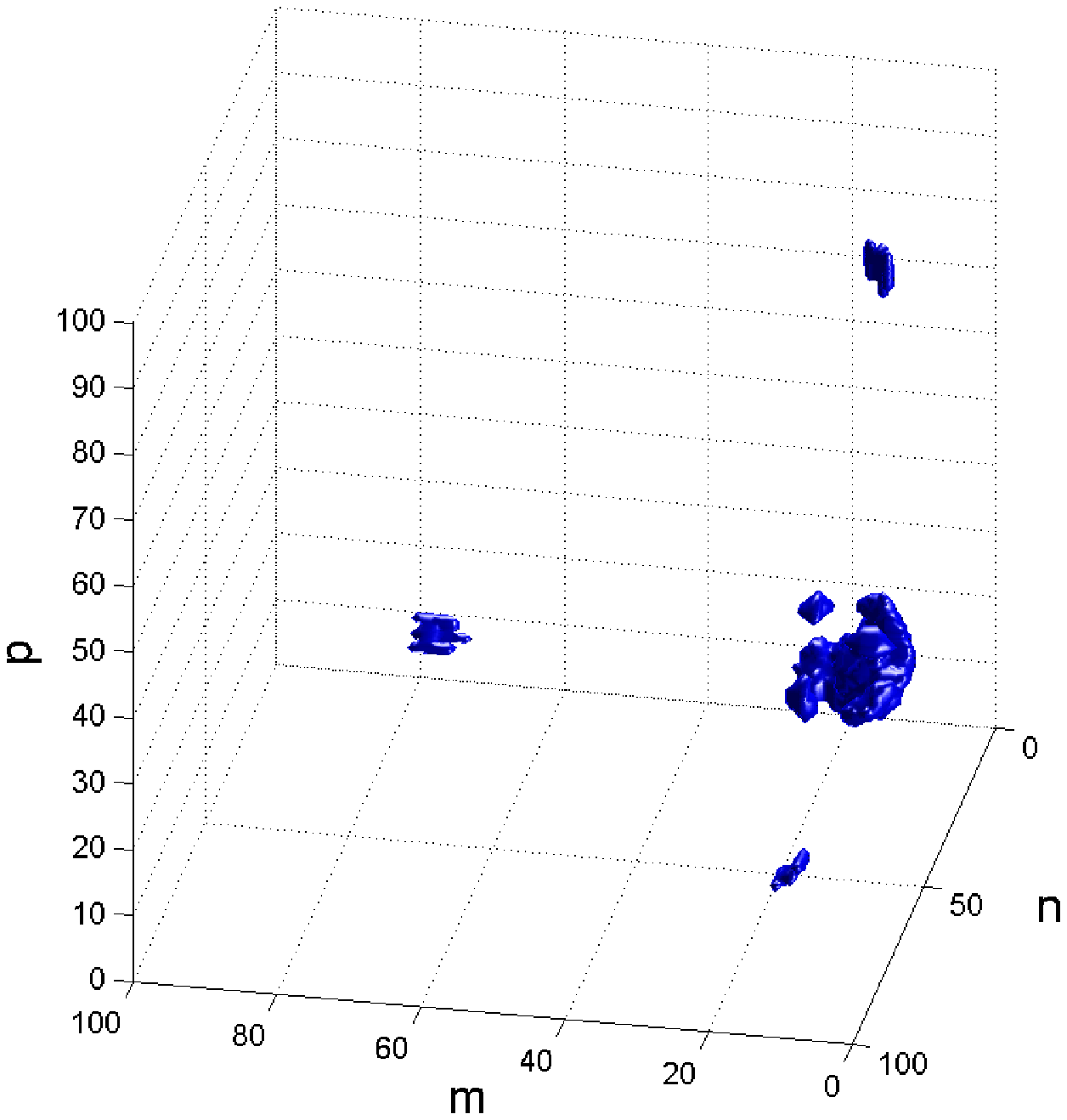}&\includegraphics[width=0.35\textwidth]{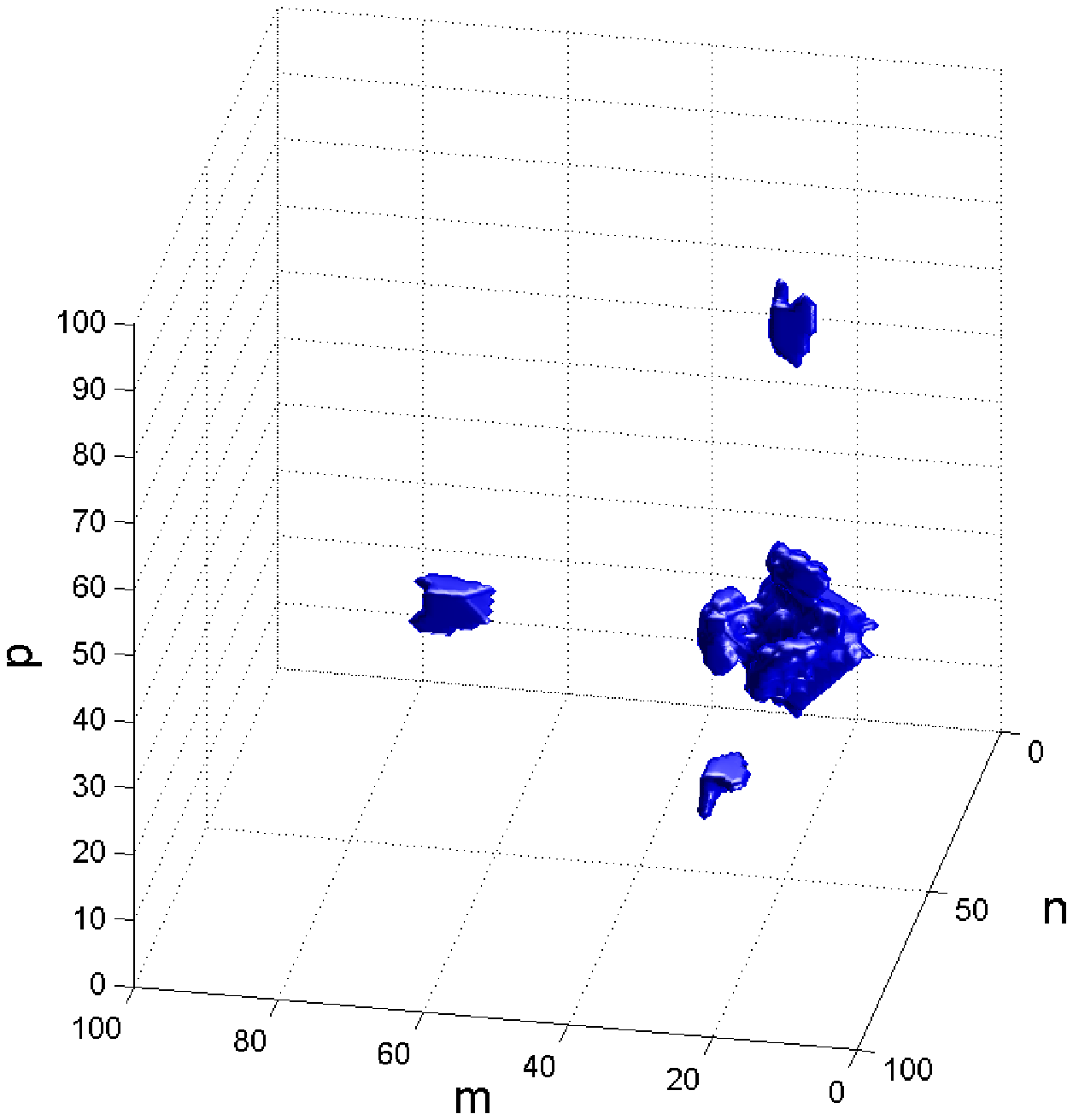}\\
$t = 10$ & $t = 12.5$\\
\includegraphics[width=0.35\textwidth]{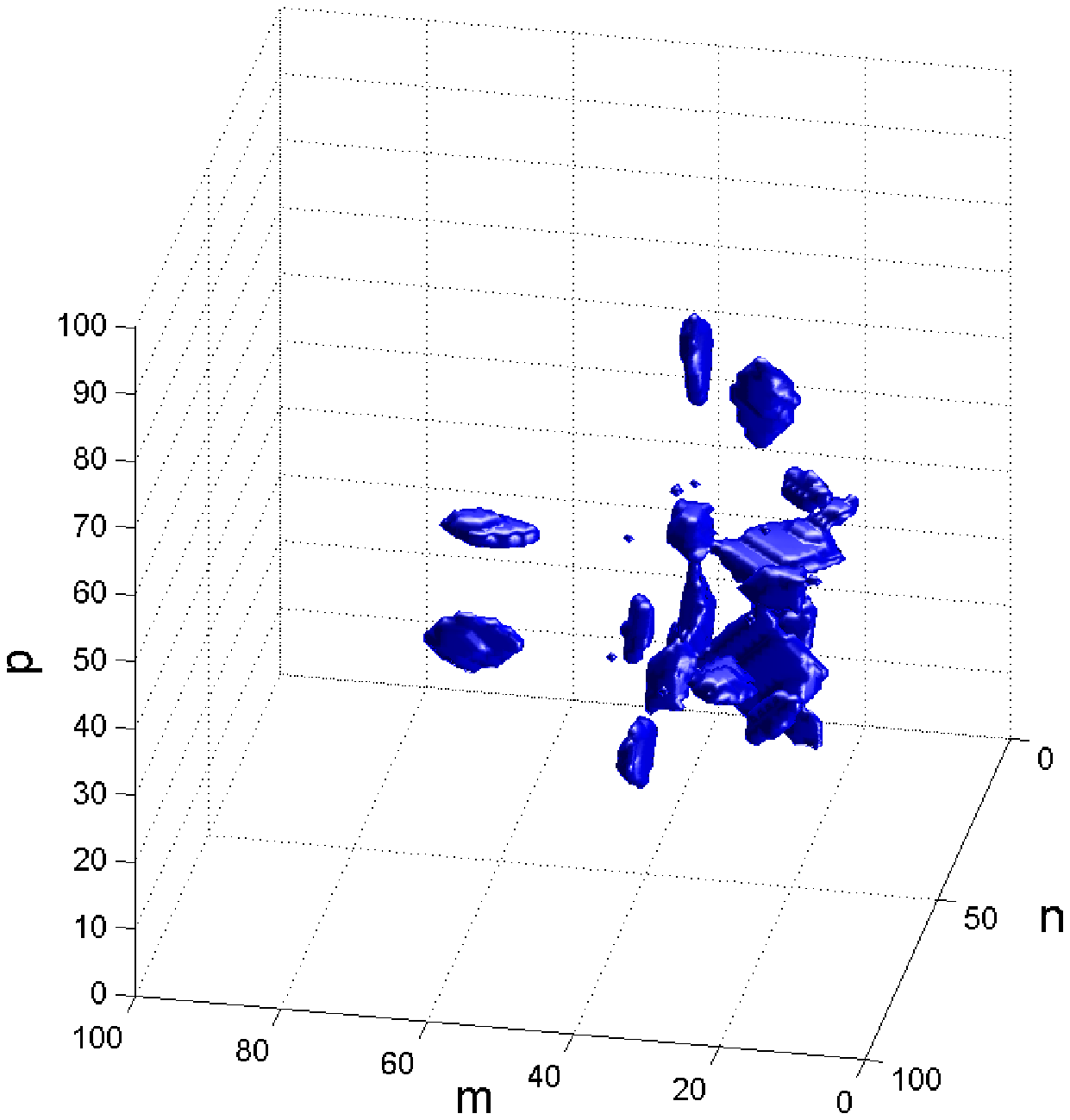}&\includegraphics[width=0.35\textwidth]{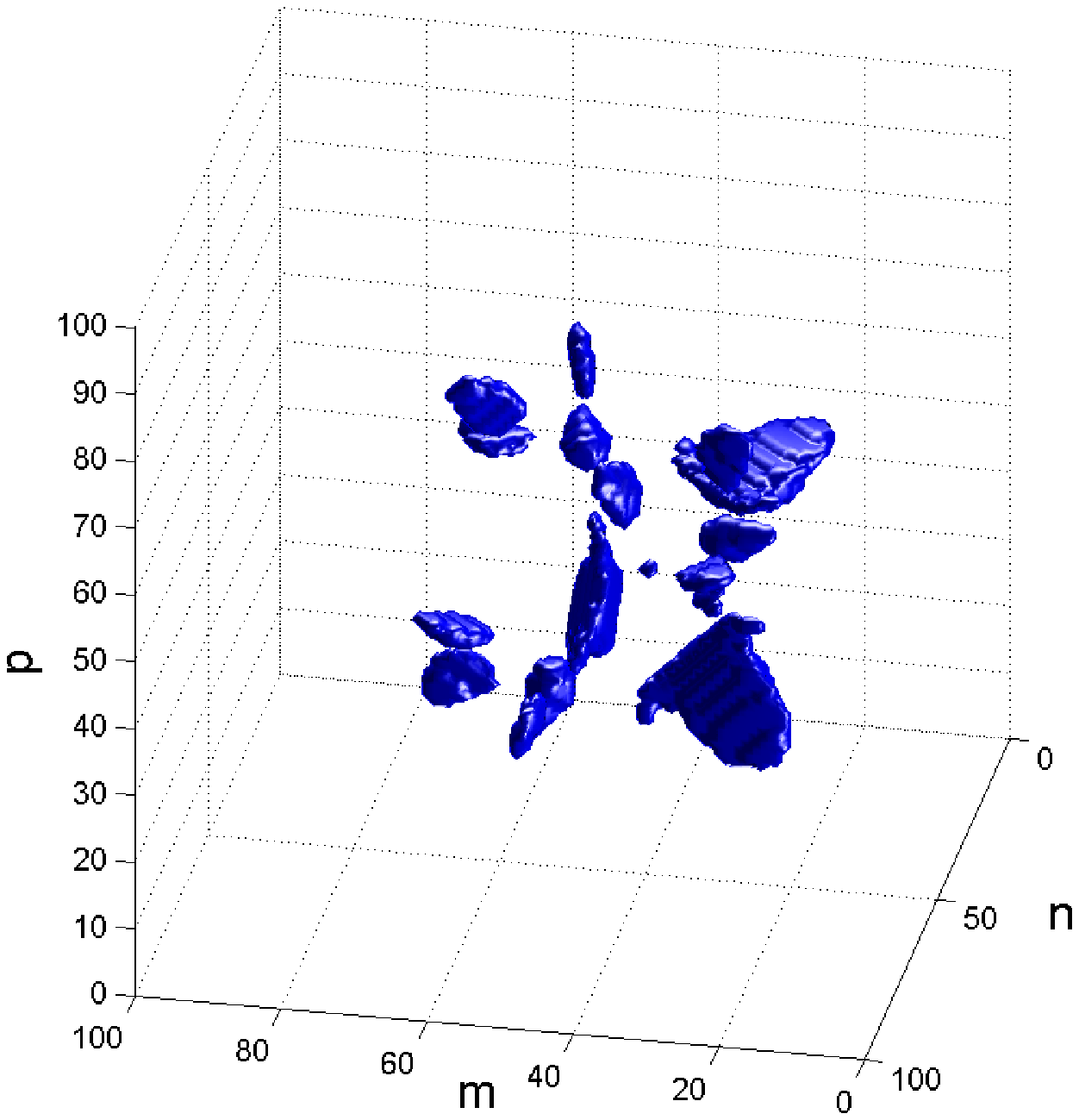}\\
$t = 15$ & $t = 17.5$\\
\includegraphics[width=0.35\textwidth]{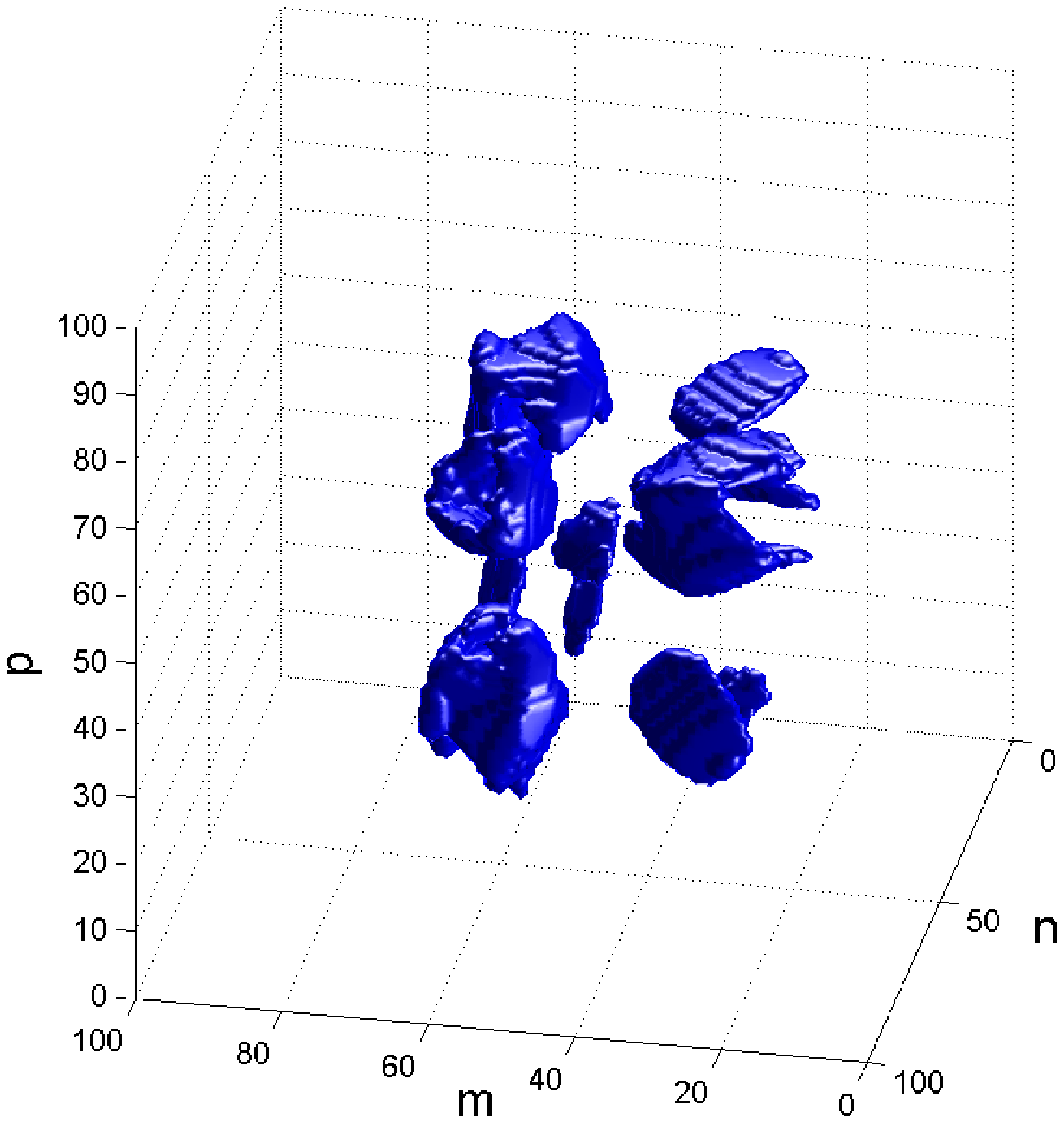}&\includegraphics[width=0.35\textwidth]{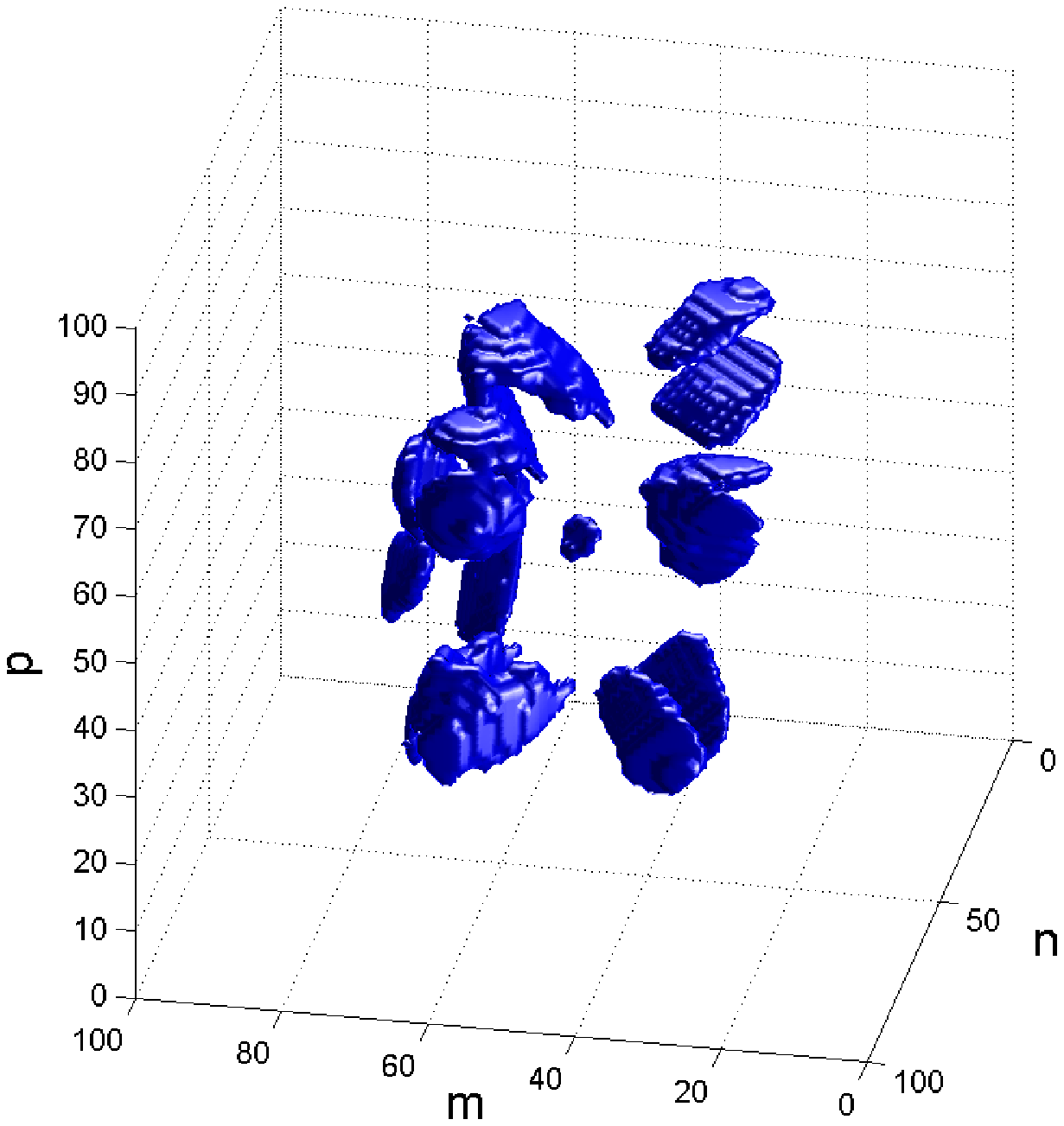}
\end{tabular}}
\caption{Time evolution of a localized solution of \eqref {Eqn:DiscreteMain} with $\gamma = 0.005$, $\beta = \mathfrak {m} ^2 = 0$ and $J = 0.01$, subject to harmonic driving with a frequency $\Omega = 0.9$ in the forbidden band-gap of the continuous-limit system. The driving function takes the form $\phi (t) = A (t) \sin (\Omega t)$ where $A$ is given by \eqref {Eqn:Amplitude}, and $b _i = 1$ for every $i = 1 , 2 , 3 , 4$. The graphs show the time evolution of regions of high energy. \label{Breathers}}
\end{figure}

We proceed to examine the sine-Gordon case next. Preliminary computational results show that, for a fixed frequency $\Omega = 0.9$ in the forbidden band-gap, the total energy of the system during a fixed period of time equal to $20$ increases smoothly with respect to $A$. Thus, we let $\Omega$ range between $0$ and $1$ and, for each pair $(\Omega , A)$, compute the associated total energy of the system. In this context, Fig. \ref {Fig12}(b) prescribes the total energy of the system versus $\Omega$ and $A$. Our results show that, the phenomenon of nonlinear supratransmission is absent in the case of radially symmetric sine-Gordon systems.

\subsection{Propagation of signals}
\label{SubSec:Signals}

The study of localized nonlinear modes in $(1 + 1)$-dimensional sine-Gordon systems is a topic of research that has produced a large amount of valuable results. Nowadays, the specialized literature in the field possesses results on this topic that range from the analytical aspects of the problem, to the numerical, to the physical, including those works where the propagation of localized modes are studied in relation with the process of nonlinear supratransmission \cite{Khomeriki,Chevrieux2,Macias-Signals}. 

\begin{figure}
\centerline{%
\includegraphics[width=0.5\textwidth]{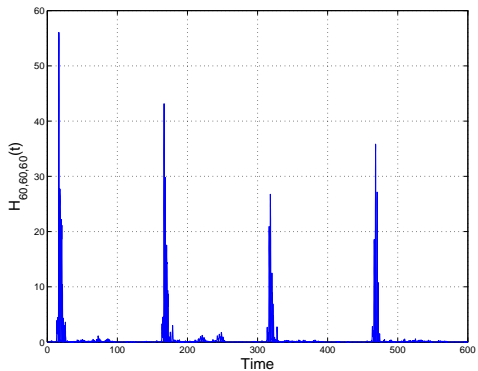}}
\caption{Time-dependent graph of local energy density $H$ of the node located at site $(60 , 60 , 60)$ in a medium described by \eqref {Eqn:DiscreteMain} with $\gamma = 0.005$, $\beta = \mathfrak {m} ^2 = 0$ and $J = 0.01$, as a consequence of driving it harmonically on the boundary at a frequency of $0.9$ in the forbidden band-gap of the continuous-limit medium, and a driving amplitude \eqref {Eqn:Amplitude} with $b _1 = b _2 = b _3 = b _4 = 1$. The peaks of the graph record the propagation of localized nonlinear modes produced by the four nonzero bits transmitted into the medium. \label{Prop}}
\end{figure}

In our study, it is particularly important to recall that supratransmission in semi-unbounded, sine-Gordon systems subject to Dirichlet harmonic driving has been characterized by the generation of moving breathers at the driving boundary once the driving amplitude has reached its critical value, and a method to control the propagation of these modes in such systems has been proposed \cite {Macias-Signals}. Thus, it seems natural to generalize this technique to the case of $(3 + 1)$-dimensional, semi-unbounded systems governed by sine-Gordon equations and subject to harmonic driving at the boundary.

Following the method proposed in \cite {Macias-Signals}, let us fix $\Omega < 1$, and let $\phi (t) = A (t) \sin (\Omega t)$. We let $P$ be a multiple of the driving period, assume that a bit $b \in \{ 0 , 1 \}$ will be transmitted into a medium \eqref {Eqn:DiscreteMain} during the period of time $[ 0 , P ]$, and suppose that $A _s$ represents the critical amplitude at which supratransmission starts for the frequency $\Omega$. Define the driving amplitude function $A _0$ as
\begin{equation}
A _0 (t ; b) = C b \left( e ^{- \Omega t / 2.5} - e ^{ - \Omega t / 0.45} \right),
\end{equation}
where $C$ is a positive real number depending on $\Omega$ that works as an amplification factor. The idea behind the definition of $A _0$ is that for the process of nonlinear supratransmission requires a certain amount of time to start to irradiate energy into the medium, during which the driving amplitude must take on values above the critical value $A _s$.

Let $\chi _A (t)$ represent the characteristic function on the set $A \subseteq \mathbb {R}$ evaluated at $t$, which is equal to $1$ if $t \in A$, and is equal to zero otherwise. In our study, we will fix $\gamma = 0.005$, $\beta = \mathfrak {m} ^2 = 0$ and $J = 0.01$, for which the value of amplitude at which supratransmission starts is $A _s = 1.41$. Numerically, we choose a step size for time equal to $0.005$, and fix a bounded cube of sides equal to $200$. Moreover, the period $P$ of signal generation will be equal to $150$, the amplification factor will be equal to $3$, and the binary sequence $(b _1 , \dots , b _k)$ will be transmitted into the medium by means of the harmonic driving function $\phi (t)$ with amplitude function defined by
\begin{equation}
A (t) = \sum _{i = 1} ^k A _0 ( t - (i - 1) P ; b _i) \chi _{[(i - 1) P , i P]} (t). \label{Eqn:Amplitude}
\end{equation}

The binary sequence $(1,1,1,1)$ will be transmitted into \eqref {Eqn:DiscreteMain} using the amplitude function just defined. In these conditions, Fig. \ref {Breathers} presents the time evolution of the local energy density of a localized solution of the medium as a result of being subject to harmonic driving with amplitude defined by \eqref {Eqn:Amplitude}, during the first period of generation of signals. The blue zones presented in these graphs are regions of high energy produced at the origin. As time evolves, these regions clearly expand and move away from the origin around the line $x = y = z = t$, for $t \in \mathbb {R} ^+$.

Finally, Fig. \ref {Prop} presents the time behavior of the local energy density at site $(60 , 60 , 60)$ in system \eqref {Eqn:DiscreteMain}, as a result of transmitting the binary code $(1 , 1 , 1 , 1)$ by means of perturbations on the boundaries. It is worth noticing that each of the peaks in the graph is a result of the localized traveling solutions --- produced by each of the nonzero bits generated at the boundaries ---, which passes by the node at site $(60 , 60 , 60)$ and moves away from the origin. Moreover, there exists a gap in time approximately equal to $150$ between two consecutive peaks, which is in perfect agreement with the value of the period $P$ of signal generation. The results evidence the possibility of accurately transmitting binary information into system \eqref {Eqn:Main} through suitable perturbations of the driving boundary.

It must be mentioned that similar results (not included here) are obtained for the approximation to the continuous case described by \eqref {Eqn:Main}, proving thus that the presence of nonlinear supratransmission in $(3 + 1)$-dimensional, dissipative sine-Gordon equations does not depend on discreteness.

\section{Conclusion}

In this work, we have presented conditionally stable, finite-difference schemes that consistently approximate the solutions to problems \eqref {Eqn:Main}, \eqref {Eqn:MainRad} and \eqref {Eqn:DiscreteMain}. Associated with these schemes, we have introduced discrete schemes to approximate consistently the local energy densities of the media and their total energy functions, in such way that the corresponding discrete rates of change of energy with respect to time consistently approximate their respective continuous rates of change of energy. In particular, if no dissipation is present and under suitable boundary conditions, the proposed methods are conservative.

Also, we have provided relevant numerical evidence that the process of nonlinear supratransmission is not present in media described by undamped radially symmetric sine-Gordon equations perturbed harmonically at the origin, proving thus that not every nonlinear system with a forbidden band-gap for the frequency in the linear dispersion relation is able to sustain this nonlinear process (contrary to a conjecture in the literature \cite {Geniet-Leon}). Our computations are supported empirically in the case of a sine-Gordon system by the fact the differential equation in \eqref {Eqn:MainRad} is approximately linear close to the origin, and analytically by the well-known fact that the origin of such systems is incapable of creating localized coherent structures. On the other hand, a similar three-dimensional medium (discrete or continuous), bounded in the first octant by the coordinate planes and subject to harmonic driving of the Dirichlet type on the boundaries, does exhibit supratransmission. Our results (presented in Section \ref {SubSec:Supra}) show a well-defined occurrence of the critical value at which supratransmission starts. 

It is interesting to notice that supratransmission in discrete $1$-dimensional chains of oscillators in achieved when the first oscillator is harmonically perturbed at a frequency in the forbidden band-gap. In a $2$-dimensional scenario, supratransmission is achieved when the boundary lines of a semi-unbounded domain are perturbed at the right frequency \cite {Chevrieux2}. Similarly, a semi-unbounded region in the $3$-dimensional case presents supratransmission when the boundary surfaces are perturbed at frequencies in the forbidden band-gap. Following this pattern, an $(n + 1)$-dimensional semi-unbounded system of oscillators described by coupled sine-Gordon equations may present supratansmission when the $n + 1$, $n$-dimensional boundaries are subject to harmonic driving with a frequency in the forbidden band-gap.

Another application to the transmission of localized nonlinear modes in $(3 + 1)$-dimensional systems governed by continuous sine-Gordon equations was provided in this work. The system was the same semi-unbounded medium studied before --- the medium governed by \eqref {Eqn:Main}, defined in the first octant and driven harmonically at the coordinate planes. By making use of nonlinear supratransmission, our results (summarized in Section \ref {SubSec:Signals}) show that a controlled propagation of wave signals can be achieved. Moreover, the propagating nodes are seen to be traveling breathers that move away from the origin on the line $x = y = z = t$, for $t \in \mathbb {R} ^+$, which is in perfect agreement with the $(1 + 1)$-dimensional scenario \cite {Macias-Signals}.

\section*{Acknowledgement}

The authors wish to thank the referees for their careful examination of the manuscript as well as for their useful comments and remarks, which led to a substantial improvement of the final product. One of us (JEMD) acknowledges support from Dr.~F.~J.~\'{A}lvarez Rodr\'{\i}guez, dean of the Faculty of Sciences of the Universidad Aut\'{o}noma de Aguascalientes, and Dr.~F.~J.~Avelar Gonz\'{a}lez, head of the Office for Research and Graduate Studies of the same university, in the form of computational resources to produce this article. The present work represents a set of partial results under project PIM08-1 at this university, and it was concluded during a visit of the author to the Tulane University of Louisiana during the winter of 2007--2008. The author also wishes to express his gratitude for the hospitality he enjoyed at Tulane.

\appendix

\section{Computational setting}

There are several important remarks on finite-difference scheme \eqref {Eqn:DiffEq2} associated with partial differential equation \eqref {Eqn:Main}. For the sake of simplification, we will assume that the spatial step-sizes $\Delta x$, $\Delta y$ and $\Delta z$ are equal, and that $N _x$, $N _y$ and $N _z$ are all equal to $N$. Moreover, we will let $J$ be equal to zero.
\begin{itemize}
\item First of all, if $\beta$ is equal to zero and $V ^\prime$ is identically equal to zero then the resulting differential equation is the damped, linear Klein-Gordon-like equation. In such case, the finite-difference method obtained is explicit.

\item If $\beta$ is equal to zero but $V ^\prime$ is not the function identically equal to zero, the partial differential equation obtained is a damped, nonlinear Klein-Gordon-like equation. In this case, method \eqref {Eqn:DiffEq2} is nonlinear and explicit; in fact, an application of Newton's method is needed to obtain the value $u _{m , n , p} ^{k + 1}$ from the known approximations at times $k - 1$ and $k$, for every $m , n , p = 1 , 2 , \dots , N$ (this was the case when performing the application in Section \ref {SubSec:Supra}).

\item Let $\beta$ be a positive real number. If $V ^\prime$ is identically equal to zero then the resulting partial differential equation is a damped, linear equation, while the finite-difference scheme associated with it is likewise linear and implicit. Meanwhile, if $V ^\prime$ is not identically equal to zero then an application of Newton's method is indispensable; we will describe this last scenario in more detail now.

Notice first of all that the algorithm of division implies that, for every positive integer $i = 1 , 2 , \dots , (N + 2) ^3$, there exist unique nonnegative integers $m$, $n$ and $p$ such that 
\begin{equation}
i = m (N + 2) ^2 + n (N + 2) + p + 1.
\end{equation}
Conversely, if $m , n , p \in \{ 0 , 1 , \dots , N + 1 \}$ then the value of $i$ given by the formula above is in the set $\{ 1 , 2 , \dots , (N + 2) ^3 \}$. Thus the term $u _{m , n , p} ^k$ may be unambiguously represented by $y _i ^k$, and the left-hand side of Eq. \eqref {Eqn:DiffEq2} will be denoted by $f _i ^k$.

Following Newton's method, the approximations at times $k - 1$ and $k$ are assumed to be known in order to compute the approximation at the $(k + 1)$st time. The Jacobian of the problem is a sparse matrix; indeed, notice that for each triplet $(m , n , p)$ with $m , n , p = 1 , 2 , \dots , N$, the following are the only nonzero partial derivatives:
\begin{equation}
\frac {\partial f _i ^k} {\partial y _{i + 1}} = \frac {\partial f _i ^k} {\partial y _{i - 1}} = \frac {\partial f _i ^k} {\partial y _{i + N + 2}} = \frac {\partial f _i ^k} {\partial y _{i - (N + 2)}} = \frac {\partial f _i ^k} {\partial y _{i + (N + 2) ^2}} = \frac {\partial f _i ^k} {\partial y _{i - (N + 2) ^2}} = \frac {\beta} {2 \Delta t \left( \Delta x \right) ^2}
\end{equation}
and
\begin{equation}
\frac {\partial f _i} {\partial y _i} (z) = \frac {1} {(\Delta t) ^2} + \frac {3 \beta} {\Delta t (\Delta x) ^2} + \frac {\gamma} {2 \Delta t} + \frac {\mathfrak {m} ^2} {2} + \frac {(z - u _i ^{k - 1}) V ^\prime (z) + V (u _i ^{k - 1}) - V (z)} {(z - u _i ^{k - 1}) ^2}.
\end{equation}
\end{itemize}

\end{document}